\documentclass[11pt]{article} 
\usepackage{a4,amssymb, amsmath, amsthm, dsfont}
\usepackage{mathrsfs,amsfonts}
\usepackage{graphics,color}
\usepackage{epsfig}

\newtheorem{theorem}{\bf Theorem}[section]

\newtheorem{conjecture}{\bf Conjecture}[section]
\newtheorem{proposition}{\bf Proposition}[section]
\newtheorem{lemma}{\bf Lemma}[section]
\def\be{\begin{equation}}
	\def\ee{\end{equation}}
\def\bse{\begin{subequations}}
	\def\ese{\end{subequations}}
\def\bge{\begin{eqnarray}}
	\def\bgee{\begin{eqnarray*}}
		\def\ege{\end{eqnarray}}
	\def\egee{\end{eqnarray*}}

\usepackage{mathtools}
\usepackage{amsfonts}
\usepackage{dsfont}
\usepackage{array}
\begin{document}

\title{Quenching for a semi-linear wave equation for MEMS}

\author{
Heiko Gimperlein\thanks{Leopold-Franzens-Universit\"{a}t Innsbruck, Engineering Mathematics, Technikerstra\ss e 13, 6020 Innsbruck, Austria} \thanks{Department of Mathematical, Physical and Computer Sciences,
University of Parma, 43124, Parma, Italy} \and  Runan He\thanks{Institut f\"{u}r Mathematik, Martin-Luther-Universit\"{a}t Halle-Wittenberg, 06120 Halle (Saale), Germany} \and  Andrew A.~Lacey\thanks{Maxwell Institute for Mathematical Sciences and Department of Mathematics, Heriot-Watt University, Edinburgh, EH14 4AS, United Kingdom}}
\date{}






\maketitle \vskip 0.5cm
\begin{abstract}
\noindent  We consider the 
formation of finite-time quenching singularities for solutions of semi-linear wave equations with negative power nonlinearities, as can model micro-electro-mechanical systems (MEMS). For radial initial data we 
obtain, formally, the existence of a sequence of quenching self-similar solutions. 
Also from formal asymptotic analysis, a solution to the PDE which is radially symmetric and increases strictly monotonically with distance from the origin 
quenches at the origin like an explicit spatially independent solution.
The latter analysis and numerical experiments suggest a detailed conjecture for the singular behaviour.
\end{abstract}

\section{Introduction}

Micro-electro-mechanical systems, or MEMS, are micro-scale devices which transform electrical energy into mechanical energy and vice versa. MEMS combine the mechanics of deformable membranes, springs or levers with electric circuits, resistors, capacitors and inductors. They are crucial components of modern technology, from smart-phones and printer heads, to airbags, micro-valves and a wide variety of sensors.

 An idealized electrostatically actuated MEMS capacitor device contains two conducting plates which, when the device is uncharged and at equilibrium, are close and parallel to each other. 


\begin{figure}[t]
	\begin{center}
		{\includegraphics[width=0.5\textwidth]{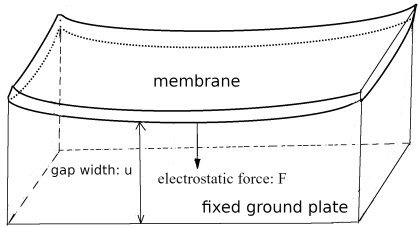}}
		\caption{Idealized electrostatically actuated MEMS capacitor.}
		\label{fig0}
	\end{center}
\end{figure}
More generally, we suppose a fixed potential difference is applied; this potential difference acts across the plates, so that the MEMS device forms a capacitor. The upper of  
the plates is flexible but pinned around its edges. The other, lower, plate is taken to be rigid and flat. The flexible plate is here taken to behave as 
a membrane,  
and it is driven by the electrostatic force between the different electric potentials and so deforms towards to the rigid plate (see Fig.~\ref{fig0}). The plate separation $u(x,t)$ can be modelled by the semi-linear equation \eqref{semi-linear-wave-equation}, below, in space dimensions $n = 1$, 2, under the
realistic assumption that: (i) the width of the gap, between the membrane and the bottom plate, is small compared to the device length; (ii) damping can be neglected.
Such a device can suffer from an instability in which the plates 
touch when the voltages are increased beyond a certain critical value. This touchdown (also known as pull-in instability) restricts the range of stable operation of the device.

In the present paper we now consider the 
semi-linear wave equation in a smoothly bounded domain $\Omega \subset \mathbb{R}^n$, $n \geq 2$:
\be\label{semi-linear-wave-equation}
\frac{\partial^2 u}{\partial t^2}=\Delta u-\frac{1}{u^2},
\ee
the case of $n=1$ having being looked at in \cite{KLNT}.
The solution $u$ of \eqref{semi-linear-wave-equation} may cease to exist after a finite time $T_{\max}$. 
The main purpose of this paper is to study the local behaviour of $u$ near $t = T_{\max}$.

The slightly more general model,
\begin{equation}\label{intro-3}
	\frac{\partial^2 u}{\partial t^2}+ \epsilon \frac{\partial u}{\partial t}-\Delta u=-\frac{\beta_F}{u^2},\quad u|_{\partial\Omega}=1,
\end{equation}
with a first-derivative term $\frac{\partial u}{\partial t}$ modelling damping, is to date, better understood. 

The pull-in voltage (the applied voltage is represented here by $\beta_F$) separates the stable operation regime for which the membrane can approach a steady state, from the “touchdown” regime for which the membrane collapses onto the rigid plate. There exists a critical value $\beta_F^{*}$ such that touchdown takes place in finite time for large $\beta_F>\beta_F^{*}$, while there is at least one stationary solution of \eqref{intro-3} for small $\beta_F<\beta_F^{*}$ \cite{FG}, see also \cite{guo2014hyperbolic}. Mathematically, the collapse, with finite-time existence occurring through the touchdown, is seen to be through quenching, $\inf_{x \in \Omega} u \to 0$ as $t \to T_{\max}$. {A point where $u$ tends to $0$ is a quenching point; $T_{\max}$ is the quenching time.}

Quenching in \eqref{intro-3} has been thoroughly studied when the damping term dominates, i.e.~$\epsilon \gg 1$ \cite{LW}. Then the semi-linear damped wave equation \eqref{intro-3} can be reduced to a parabolic equation, for which general techniques based on the maximum principle have been developed. Specifically for the quenching profile we refer to \cite{GuoJ}, as well as \cite{boni2000quenching,drosinou2020impacts,ghoussoub2008estimates,guo2012nonlocal,kavallaris2008touchdown}.

This article studies the quenching profile when the contribution of the inertial terms dominates, i.e.~$\epsilon \ll 1$. In this regime, \eqref{intro-3} simplifies to the hyperbolic model
\begin{equation}\label{intro-5}
	\frac{\partial^2 u}{\partial t^2}-\Delta u=-\frac{\beta_F}{u^2},\quad u|_{\partial\Omega}=1.
\end{equation}
Local and global existence results for \eqref{intro-5} and the dynamical behaviour of solutions, depending on the parameter $\beta_F$, have been obtained in \cite{TNLK}. The analysis of quenching in semi-linear wave equations goes back to \cite{ChL}, who showed, in our language, the existence of a critical value $\beta_F^{*}$ separating quenching from stationary behaviour in one dimension. While qualitative results are now well-understood for a variety of hyperbolic models \cite{FG, TNLK,levine1984abstract,LZ,smith1989hyperbolic}, the behaviour close to quenching, i.e.~the quenching profile, has not been studied previously for higher dimensions $n\geq2$.

{To understand the local behaviour of $u$ near the quenching time, we consider the existence of self-similar, radial solutions to \eqref{semi-linear-wave-equation} in $\mathbb{R}^n$. We look for these self-similar solutions using new variables:
\begin{equation*}
u(x,t)=\left(T_{\max}-t\right)^{\frac{2}{3}}v(\eta), \quad \eta=\frac{|x|}{T_{\max}-t}.
\end{equation*}
{The form $u=\left(T_{\max}-t\right)^{\frac{2}{3}}v$ is indicated by the spatially uniform solution of \eqref{semi-linear-wave-equation}.} 
For $v$ to indeed give a solution $u$ at quenching, i.e. as $t \to T_{\max}$ so $\eta \to \infty$ for fixed $|x|$, $v$ should behave as const.$\times \eta^{\frac23}$ for large $\eta$. A key aim of the present paper is thus to investigate the existence of smooth $v$ having the right growth condition at infinity, such that $u$ satisfies \eqref{semi-linear-wave-equation} in $\mathbb{R}^n$.}
	
{In dimension $n=1$, reference \cite{TNLK} shows that the {trivial, constant  $v_0 \equiv c^*=\left(\frac{9}{2}\right)^{\frac{1}{3}}$ is the only smooth global  $v$.} In this article we show that in dimensions $2 \le n \le 7$, however, non-trivial radial smooth quenching solutions exist in $\mathbb{R}^n$: There is a sequence of non-trivial smooth global solutions $v_j$ with $v_j(0) \to 0^+$, which satisfy the growth condition at infinity. }

{A precise statement of our results is given in Theorem \ref{ATM5} in Section \ref{GloExist}. }

Results related to those of Section \ref{GloExist}
have been proved for blow-up of solutions to semi-linear PDEs
with nonlinearities $u^p$ in \cite{dai2021} for wave equations and
in \cite{collot2016stability} for heat equations. Although the formal results
of the present paper are closely related to those of those two works, being
somewhat weaker, the structure of {the relevant self-similar equation} \eqref{SelfSimilarEquation} differs
from that of the corresponding ODEs for papers \cite{dai2021} and \cite{collot2016stability}:
a global solution is not guaranteed {\it a-priori} {in our case.} This means that key steps do not apply,
and the Lyapunov function used in \cite{dai2021} might not exist here.
Other related works on blow-up are referenced in \cite{collot2016stability} and \cite{dai2021}.

{The relevance of the global self-similar solutions $v_j$ for quenching in \eqref{intro-5} hinges on their stability. In Section~\ref{sectionasymp} we go into this by
showing formal and numerical evidence that in $n$ dimensions the trivial self-similar solution $v_0 = c^*$ is stable, {other than to a very particular class of (spatially uniform) perturbations which correspond to solutions quenching at different times.} The generic local behaviour of a strictly increasing radial quenching solution $u$ near the quenching time $T_{\max}$ is given by $u(x,t) = \left(T_{\max}-t\right)^{\frac{2}{3}}V(\xi,t)$, with $\xi=\frac{r}{(T_{\max}-t)^{1/2}}$, {$r=|x|$ being the distance to the quenching point,} and
\[
	V(\xi,t) \sim C\xi^{\frac{4}{3}} \quad \text{{as $\xi \to \infty$} and } \quad V(0,t) \to c^*\notag
\]
for $t \to T_{\max}$. Correspondingly,
\[
u(x,T_{\max}) \sim C r^{\frac{4}{3}} \quad \text{as $r \to 0$ \quad with }\quad
u(0,t) \sim c^* (T_{\max}- t)^{\frac 23} \text{ as $t \to T_{\max}$,}
\]
provided $u$ is a strictly increasing function of  $r$ near the quenching time. Here $C$ is a constant which depends on $n$, and on the initial and boundary conditions for the semi-linear wave equation. 
We conjecture that for $2 \le n \le 7$, the smooth self-similar solutions $v_j$ are unstable for $j \in \mathbb{N}$, as supported by the formal and numerical results in the paper \cite{TNLK} for the problem in one space dimension. }

{A precise statement is given in Conjecture \ref{mainconj} in Section~\ref{sectionasymp}.}

We discuss the importance of our results in the concluding Section~\ref{Disc}.

\section{Global Existence of Self-Similar Solutions}\label{GloExist}

{To understand the local behaviour of the solution $u$ to \eqref{semi-linear-wave-equation} near the quenching time, we consider the existence of self-similar, radial solutions to the equation in $\mathbb{R}^n$:
\be\label{semi-linear-wave-equation-R}
\frac{\partial^2 u}{\partial t^2}=\frac{\partial^2 u}{\partial r^2}+\frac{n-1}{r}\frac{\partial u}{\partial r}-\frac{1}{u^2}.
\ee
Here, $r=|x|$ and $u=u(r, t)$. We can look for self-similar solutions using new variables:
\be\label{SelfSimilarVariables}
u(r,t)=\left(T_{\max}-t\right)^{\frac{2}{3}}v(\eta), \quad \eta=\frac{r}{T_{\max}-t}.
\ee
From  \eqref{semi-linear-wave-equation-R},  
$v$ then satisfies the self-similar equation
\be\label{SelfSimilarEquation}
(1-\eta^2)\frac{d^2v}{d\eta^2}+\left(\frac{n-1}{\eta}-\frac{2}{3}\eta\right)\frac{dv}{d\eta}+\frac{2}{9}v-\frac{1}{v^2}=0.
\ee
It can be noted, from the transformations \eqref{SelfSimilarVariables}, that for $v$ to indeed give a solution $u$ at quenching, i.e. as $t \to T_{\max}$ so $\eta \to \infty$ for fixed $r$, $v$ should behave as const.$\times \eta^{\frac23}$ for large $\eta$. A key aim of the present paper is thus to investigate the existence of smooth solutions of \eqref{SelfSimilarEquation} having the right growth condition at infinity.}

{Our analysis in this section will be based on a shooting method: imposing initial conditions at $\eta = 0$,
\be\label{InitialValue0}
v(0)=c>0,\quad \frac{dv}{d\eta}(0)=0,
\ee
we determine those $c$ for which $v$ extends smoothly across $\eta=1$ and has the physically relevant behaviour $v \sim $ const.$\times \eta^{2/3}$ as $\eta \to \infty$. }

{Note that \eqref{SelfSimilarEquation} admits the trivial, constant solution $v_0 \equiv c^*=\left(\frac{9}{2}\right)^{\frac{1}{3}}$. In dimension $n=1$, reference \cite{TNLK} shows that this is the only smooth global solution of \eqref{SelfSimilarEquation}. In dimensions $2 \le n \le 7$, however, non-trivial global smooth solutions exist:}
{\begin{theorem}[Global Existence of Self-Similar Solutions, Formal Theorem]\label{ATM5}
a)	There exists a discrete set of initial conditions $c = c_j < c^*$ in \eqref{InitialValue0}, with $c_j \to 0^+$ as $j \to \infty$, such that the self-similar equation \eqref{SelfSimilarEquation} has a non-trivial smooth global solution $v_j$ for $2 \le n \le 7$. \\
b) The solutions $v_j(\eta) \sim a \eta^{\frac{2}{3}}$ for $c_j^{\frac{3}{2}}\ll\eta\ll1$, where $a=\left[\frac{2}{3}\left(n-\frac{4}{3}\right)\right]^{-\frac{1}{3}}$, and $v_j(\eta) \sim b_{j} \eta^{\frac{2}{3}}$ for $\eta \to \infty$, where $b_{j}$ is a constant depending upon $n$ and $j$.
\end{theorem}
Theorem \ref{ATM5} is proved formallly in this section below.}

{Formal asymptotics also indicate that the smooth solutions $v_j(\eta)$ are monotonically increasing with respect to $\eta$ for large enough $j$.} 

{Note that \eqref{SelfSimilarEquation} presents challenges through the degeneracy of the coefficient of the highest-order derivative at $\eta=1$, the light cone, and the possibility of limited existence of solution, as $v$ could vanish at some finite value of $\eta$.}

{We briefly remark that, for a steady state, equation \eqref{semi-linear-wave-equation-R} reduces to a radial ODE  \eqref{EFODE} 
for which the qualitative behaviour has been studied extensively:
\begin{equation}\label{EFODE}
	\frac{d}{d\eta}\left(\eta^{n-1}\frac{dv}{d\eta}\right)+\lambda\eta^{n-1}  \mathcal{F}(v)=0 ,
\end{equation} 
here writing $v=u$ and $\eta = r$, with nonlinearity $\mathcal{F}(v) = -1/v^2$, similar to an equation of Emden-Fowler type.
For $n>2$ equation \eqref{EFODE} has been considered in the classical work \cite{JL}, establishing the local behaviour of a singular solution. For existence, nonexistence and multiplicity of solutions we refer to the seminal paper by Gelfand \cite{gel1959some}.}

\subsection{Asymptotic Behaviours of Self-Similar Solutions}
To prove Theorem \ref{ATM5} formally we must understand the asymptotic behaviours of the self-similar solutions, as indicated by the following Proposition \ref{ATM1}, Lemma \ref{ATM3} and Lemma \ref{ATM4}.

\subsubsection*{Asymptotic Behaviour of Self-Similar Solution Near 0}
\begin{proposition}\label{ATM1}
	\ For $0<c\ll\left(\frac{9}{2}\right)^{\frac{1}{3}}$, the asymptotic behaviour of the solution to the self-similar equation
	\eqref{SelfSimilarEquation} with the initial values \eqref{InitialValue0} for $c^{\frac{3}{2}}\ll\eta\ll1$ is of the form, when $2 \le n \le 7$,
	\begin{align}
		v&\sim a\eta^{\frac{2}{3}}+A_1c^{\frac{3}{4}n-\frac{1}{2}}\eta^{1-\frac{n}{2}}\cos\left(b\ln(\eta c^{-\frac{3}{2}})+A_2\right) + \dots ; 
\label{AsymForm1}
	\end{align}
	and when $n\geq8$,
	\begin{align}
		v&\sim a\eta^{\frac{2}{3}}+c^{\frac{3}{4}n-\frac{1}{2}}\eta^{1-\frac{n}{2}}\left(A_3c^{-\frac{3}{2}b}\eta^b+A_4c^{\frac{3}{2}b}\eta^{-b}\right) + \dots 
\label{AsymForm2}.
	\end{align} 
Here, $a=\left[\frac{2}{3}\left(n-\frac{4}{3}\right)\right]^{-\frac{1}{3}}$, $b = \left|\left(\frac{n}{2}-\frac{7}{3}\right)^2-\frac{8}{3}\right|^{1/2} \geq 0$, and $\{A_j\}_{j=1}^4$ are suitable constants dependent on $n$, but not $c$.
\end{proposition}

\subsubsection*{Asymptotic Behaviour of Self-Similar Solution Near 1}
\begin{lemma}\label{ATM3}
	{Any positive solution $v$ of the equation \eqref{SelfSimilarEquation} 			admits a decomposition into a smooth  and a singular part, $v = {v_{reg}}+v_{sing,\pm}$, at $\eta=1^\pm$. We have $$v_{sing,\pm}(\eta) =  \textstyle{\sum_{\alpha \in  E}c_\alpha^\pm |\eta-1|^\alpha}\ ,$$ where $c_\alpha^\pm$ are constants and $E$ is specified below \eqref{AsymPowerSeries}. The  smallest exponent $\alpha \in E$ is given by $\alpha=\frac{n}{2}+\frac{1}{6}$.}
	\end{lemma}
	\subsubsection*{Asymptotic Behaviour of Self-Similar Solution Near Infinity}
	\begin{lemma}\label{ATM4}
		A solution $v$ of the equation \eqref{SelfSimilarEquation} existing for large $\eta$ is either asymptotically constant $v(\eta) \to c^\ast$ for $\eta \to \infty$, or it has an asymptotic form 
		\begin{equation}\label{AsymInfty}
			\textstyle{v(\eta)\sim \sum_{\alpha \in E}b_\alpha\eta^{\alpha} = b_{\frac{2}{3}} \eta^{\frac{2}{3}} + \dots,  b_{\frac{2}{3}} > 0,}
		\end{equation}
		for $\eta$ near $\infty$. The dots denote lower-order terms.  
		
	\end{lemma}

\subsubsection*{Proof of Proposition \ref{ATM1}}
	For small $0<c\ll\left(\frac{9}{2}\right)^{\frac{1}{3}}$, we set
	\[v=cV, \quad \eta=c^{\frac{3}{2}}\bar{r},\quad V=V(\bar{r}),\quad v=v(\eta).\]
	From \eqref{SelfSimilarEquation}, we obtain
	\be\label{A-SelfSimilarEquation2}
	\textstyle{\frac{d^2V}{d\bar{r}^2}+\frac{n-1}{\bar{r}}\frac{dV}{d\bar{r}}-\frac{1}{V^2}=c^3\left(\bar{r}^2\frac{d^2V}{d\bar{r}^2}+\frac{2}{3}\bar{r}\frac{dV}{d\bar{r}}
		-\frac{2}{9}V\right),\quad V(0)=1, \quad \frac{dV}{d\bar{r}}(0)=0.}
	\ee
	Decomposing $V$ as
	\be\label{Decom}
	V=V_0+c^3V_1,
	\ee
	the leading-order term $V_0$ is a solution of the initial-value problem
	\be\label{A-SelfSimilarEquation3}
	\textstyle{\frac{d^2V_0}{d\bar{r}^2}+\frac{n-1}{\bar{r}}\frac{dV_0}{d\bar{r}}-\frac{1}{V_0^2}=0, \quad V_0(0)=1, \quad \frac{dV_0}{d\bar{r}}(0)=0.}
	\ee
	A solution to an ODE like that of \eqref{A-SelfSimilarEquation3} has been extensively studied in terms of phase plane analysis in the paper \cite{JL} of Joseph and Lundgren. 

	Rescale the variables $V_0$ and $\bar{r}$ by
	\be\label{RescaledVariables}
	V_0=\bar{r}^{\frac{2}{3}}Z_0,\quad \bar{r}=e^{\bar{s}},\quad \bar{s}=\ln\bar{r}.
	\ee
Note that the initial conditions in \eqref{A-SelfSimilarEquation3} give
\be \label{ICZ0}
Z_0 \sim e^{- \frac 23 \bar s} \quad \mbox{ and } \quad \frac{d Z_0}{d \bar s} \sim - \frac 23 e^{- \frac 23 \bar s} \quad \mbox{ as } \bar s \to \infty .
\ee
	The differential equation of \eqref{A-SelfSimilarEquation3} in $\bar{r}>0$ transforms into the following equation in $\bar{s} \in \mathbb{R}$:
	\be\label{A-SelfSimilarEquation6}
	\textstyle{\frac{d^2Z_0}{d\bar{s}^2}+\left(n-\frac{2}{3}\right)\frac{dZ_0}{d\bar{s}}+\frac{2}{3}\left(n-\frac{4}{3}\right)Z_0=\frac{1}{Z_0^2}.}
	\ee
	We rewrite \eqref{A-SelfSimilarEquation6} as a system
	\be\label{A-SelfSimilarEquation6-S}
	\frac{d}{d\bar{s}}\begin{pmatrix}Z_0\\  Z_0'\end{pmatrix}=\begin{pmatrix}0\  & 1\\  \frac{2}{3}\left(\frac{4}{3}-n\right)\ & \frac{2}{3}-n\end{pmatrix}
	\begin{pmatrix}Z_0\\  Z_0'\end{pmatrix}+\begin{pmatrix}0\\  Z^{-2}_0\end{pmatrix} .
	\ee
	An equilibrium point of the system \eqref{A-SelfSimilarEquation6-S} is given by the equation
	\[\frac{d}{d\bar{s}}\begin{pmatrix}Z_0\\ Z_0'\end{pmatrix}=\begin{pmatrix}0\\ 0\end{pmatrix}.\]
	$(Z_0, Z_0')=(a, 0)$ is a possible solution, and an equilibrium point of the system \eqref{A-SelfSimilarEquation6-S} in the phase plane, provided $a=\left[\frac{2}{3}\left(n-\frac{4}{3}\right)\right]^{-\frac{1}{3}}$. Correspondingly, $a\eta^{\frac{2}{3}}$ is a possible solution of \eqref{A-SelfSimilarEquation6}. 

It is easily seen that in fact $(a,0)$ is the only equilibrium of \eqref{A-SelfSimilarEquation6-S},
that the solution of \eqref{A-SelfSimilarEquation6} satisfying condition \eqref{ICZ0}
has $Z_0 > 0$ for all $\bar s$, and that \eqref{A-SelfSimilarEquation6} has the
Lyapunov function $\displaystyle \frac 12 \left( \frac{d Z_0}{d \bar s} \right)^2 + 
\frac 13 \left( n - \frac 43 \right) Z_0^2 + \frac 1{Z_0}$. It is thus clear
that $(Z_0,Z_0') \to (a,0)$ as $\bar s \to \infty$.

We linearize \eqref{A-SelfSimilarEquation6-S} around the stable point $(Z_0, Z_0')=(a, 0)$ and obtain
	\be\label{A-SelfSimilarEquation6-SL}
	\frac{d}{d\bar{s}}\begin{pmatrix}Z_0\\  Z_0'\end{pmatrix}=\begin{pmatrix}0\  & 1\\  \frac{2}{3}\left(\frac{4}{3}-n\right)-\frac{2}{a^3}\ & \frac{2}{3}-n\end{pmatrix}\begin{pmatrix}Z_0\\  Z_0'\end{pmatrix}:=\mathcal{A}\begin{pmatrix}Z_0\\  Z_0'\end{pmatrix} .
	\ee
	The characteristic equation of the matrix $\mathcal{A}$ is given by
\begin{equation} \label{chareq}
\textstyle{\lambda\left(\lambda-\frac{2}{3}+n\right)-\frac{2}{3}\left(\frac{4}{3}-n\right)+\frac{2}{a^3} = \lambda\left(\lambda-\frac{2}{3}+n\right) + 2 \left(n - \frac{4}{3}\right) =0}.
\end{equation}
	Its roots give us the following for the eigenvalues of  $\mathcal{A}$
	\[\lambda_{\pm}=\textstyle{\frac{1}{3}-\frac{n}{2}\mp\sqrt{\Delta},\quad \Delta=\left(\frac{n}{2}-\frac{7}{3}\right)^2-\frac{8}{3}.}\]
	If $2 \le n \le 7$, then $\Delta<0$, $\lambda_{+}$ and $\lambda_{-}$ are complex eigenvalues with negative parts. Thus $(Z_0, Z_0')=(a, 0)$ is a stable spiral point.
	
	If $n\geq8$, then  $\Delta\geq0$, $\lambda_{+}$ and $\lambda_{-}$ are real and negative eigenvalues, hence  $(Z_0, Z_0')=(a, 0)$ is a stable node.
	
	The asymptotic form of a solution $Z_0$ to the equation \eqref{A-SelfSimilarEquation6} near the steady state is given by $Z_0\sim a+\widetilde{Z}_0$ for $\bar{s}\to\infty$, where $\widetilde{Z}_0=\widetilde{Z}_0(\bar{s})$ satisfies
	\be\label{A-SelfSimilarEquation8}
	\textstyle{\frac{d^2\widetilde{Z}_0}{d\bar{s}^2}+\left(n-\frac{2}{3}\right)\frac{d\widetilde{Z}_0}{d\bar{s}}+2\left(n-\frac{4}{3}\right)\widetilde{Z}_0=0.}
	\ee
	Solutions of \eqref{A-SelfSimilarEquation8} are of the form
	\begin{equation}\label{Z0}
		\widetilde{Z}_0(\bar{s})=Ae^{-l_+\bar{s}}+Be^{-l_-\bar{s}},
	\end{equation}
	where $A$ and $B$ are arbitrary constants and $l_+$ and $l_-$ are the roots of
	\[\textstyle{l_{\pm}^2-\left(n-\frac{2}{3}\right)l_{\pm}+2\left(n-\frac{4}{3}\right)=0.}\]
	Then $l_{\pm}=-\lambda_{\pm}$,
	\[l_{\pm}=\textstyle{\frac{n}{2}-\frac{1}{3}\pm\sqrt{\Delta}=\frac{n}{2}-\frac{1}{3}\pm bi,\quad \text{when}\ \frac{4}{3}<n<8,}\]
	\[l_{\pm}=\textstyle{\frac{n}{2}-\frac{1}{3}\pm\sqrt{\Delta}=\frac{n}{2}-\frac{1}{3}\pm b},\quad \text{when}\ n\geq8.\]
	Here, $b=\sqrt{|\Delta|}\geq0$ and $\Delta=\left(\frac{n}{2}-\frac{7}{3}\right)^2-\frac{8}{3}$. Hence, we obtain the asymptotic forms of $Z_0$
	\bse\label{Asym-Z}
	\be\label{Asym-Z1}
	Z_0(\bar{s})\sim a+A_1e^{\left(\frac{1}{3}-\frac{n}{2}\right)\bar{s}}\cos\left(b\bar{s}+A_2\right),\quad\text{when}\  2 \le n \le 7,
	\ee
	\be\label{Asym-Z2}
	Z_0(\bar{s})\sim a+A_3e^{\left(b+\frac{1}{3}-\frac{n}{2}\right)\bar{s}}+A_4e^{\left(-b+\frac{1}{3}-\frac{n}{2}\right)\bar{s}},\quad \text{when}\ n\geq8,
	\ee
	\ese
	for $\bar{s}\to\infty$. 
	Here, constants $A_i$, $i=1,\ 2,\ 3,\ 4,$ are independent of $c$ since both differential equation \eqref{A-SelfSimilarEquation6} and initial conditions \eqref{ICZ0} involve only the dimension $n$.

	Because $V \sim   (\bar{r})^{\frac{2}{3}}  Z_0$, and $\bar{r}=e^{\bar{s}}$, then, for $\bar r \gg 1$,
	\begin{align}
		V\sim &a(\bar{r})^{\frac{2}{3}}+A_1(\bar{r})^{1-\frac{n}{2}}\cos\left(b\ln\bar{r}+A_2\right) \, , 
\, \, \text{when}\, \,  2 \le n \le 7, \label{Asym-Y14}
	\end{align}
	\begin{align}
		V\sim &a(\bar{r})^{\frac{2}{3}}+(\bar{r})^{1-\frac{n}{2}}+\left(A_3(\bar{r})^b+A_4(\bar{r})^{-b}\right)  
		,\quad \text{when}\ n\geq8.\label{Asym-Y15}
	\end{align}
	Since $cV=v$ and $\eta=c^{\frac{3}{2}}\bar{r}$, we obtain the asymptotic form of $v$, when $2 \le n \le 7$,
	\begin{align}
		v\sim& \textstyle{a\eta^{\frac{2}{3}}+A_1c^{\frac{3}{4}n-\frac{1}{2}}\eta^{1-\frac{n}{2}}\cos\left(b\ln(\eta c^{-\frac{3}{2}})+A_2\right)}  , 
\label{Asym-v1}
	\end{align}
	and when $n\geq8$,
	\begin{align}
		v\sim& a\eta^{\frac{2}{3}}+c^{\frac{3}{4}n-\frac{1}{2}}\eta^{1-\frac{n}{2}}\left(A_3c^{-\frac{3}{2}b}\eta^b+A_4c^{\frac{3}{2}b}\eta^{-b}\right) . 
\label{Asym-v2}
	\end{align} 
	This concludes our discussion of the asymptotics in Proposition \ref{ATM1}.\\

	\subsubsection*{Proof of Lemma \ref{ATM3}}
	
	{Without loss of generality, we consider the case $\eta>1$. Let $v$ be a positive solution of the equation \eqref{SelfSimilarEquation}, and $v(1)=c_0^+>0$.}
	
	{We set $\hat{\eta}=\eta-1$, $\hat{v}(\hat{\eta})=v(\eta)$. Then $\hat{v}$ satisfies
		\be\label{RescaledEqnForAym1}
		\textstyle{-\hat{\eta}(\hat{\eta}+2)\frac{d^2\hat{v}}{d\hat{\eta}^2}+\left(\frac{n-1}{\hat{\eta}+1}-\frac{2}{3}(\hat{\eta}+1)\right)\frac{d\hat{v}}{d\hat{\eta}}+
			\frac{2}{9}\hat{v}-\frac{1}{\hat{v}^2},\quad \hat{v}(0)=C_1.}
		\ee
		The asymptotic expansion of the solution to \eqref{RescaledEqnForAym1} takes the form 
		\be\label{AsymPowerSeries}
		\textstyle{\hat{v}(\hat{\eta})\sim\sum_{\alpha \in E\cup \mathbb{N}_0}c_\alpha^+\hat{\eta}^{\alpha} = \hat{v}_{reg,+}(\hat{\eta}) + \hat{v}_{sing,+}(\hat{\eta}),}
		\ee
		where $\hat{v}_{sing,+}(\hat{\eta})\sim \sum_{\alpha \in E}c_\alpha^+\hat{\eta}^{\alpha}$. Here, $E \subset \mathbb{C}$ with $\min\{\mathrm{Re }\ \alpha : \alpha \in E\} >0$, and $c_{0}^+ > 0$ as above.   }
	
	{To determine the smallest ${\underline{\alpha}} \in E$, we substitute \eqref{AsymPowerSeries} into \eqref{RescaledEqnForAym1} and obtain
		\begin{align}
			&-\hat{\eta}\left(\hat{\eta}+2\right)\left(2\hat{c}_2^++\dots+{\underline{\alpha}}({\underline{\alpha}}-1)\hat{c}_{\underline{\alpha}}^+\hat{\eta}^{{\underline{\alpha}}-2}+[{\underline{\alpha}}+1]{\underline{\alpha}}\hat{c}_{{\underline{\alpha}}+1}^+\hat{\eta}^{{\underline{\alpha}}-1}+\dots\right)\notag\\
			&\hspace*{-0.1cm}+\hspace*{-0.15cm}\textstyle{\left[(n-1)(1-\hat{\eta}+\hat{\eta}^2\hspace*{-0.1cm}+\hspace*{-0.1cm}\dots\hspace*{-0.1cm})\hspace*{-0.1cm} - \hspace*{-0.1cm}\frac{2}{3}(\hat{\eta}+1)\right]}\hspace*{-0.1cm}(\hat{c}_1^++2\hat{c}_2^+\hat{\eta}\hspace*{-0.1cm}+\hspace*{-0.1cm}\dots\hspace*{-0.1cm}+{\underline{\alpha}}\hat{c}_{\underline{\alpha}}^+\hat{\eta}^{{\underline{\alpha}}-1}+({\underline{\alpha}}+1)\hat{c}_{{\underline{\alpha}}+1}^+\hat{\eta}^{{\underline{\alpha}}}+\hspace*{-0.1cm}\dots\hspace*{-0.05cm})\notag\\
			&\textstyle{+\frac{2}{9}\left(\hat{c}_0^++\hat{c}_1^+\hat{\eta}+\hat{c}_2^+\hat{\eta}^2+\dots+\hat{c}_{\underline{\alpha}}^+\hat{\eta}^{\underline{\alpha}}+\hat{c}_{{\underline{\alpha}}+1}^+\hat{\eta}^{{\underline{\alpha}}+1}+\dots\right)}\notag\\
			&\textstyle{-(\hat{c}_0^+)^{-2}+\frac{2\hat{c}_1^+}{(\hat{c}_0^+)^3}\hat{\eta}+\frac{2\hat{c}_2^+}{(\hat{c}_0^+)^3}\hat{\eta}^2
				+\dots+\frac{2\hat{c}_{\underline{\alpha}}^+}{(\hat{c}_0^+)^3}\hat{\eta}^{{\underline{\alpha}}}+\frac{2\hat{c}_{{\underline{\alpha}}+1}^+}{(\hat{c}_0^+)^3}\hat{\eta}^{{\underline{\alpha}}+1}+\dots=0}\notag.
		\end{align}
		Dots indicate higher-order terms. }
	
	{In order to find the  smallest non-integer exponent $\underline{\alpha}$ in the set $E$, we balance the terms of $\hat{\eta}^{{\underline{\alpha}}-1}$ and get $-2{\underline{\alpha}}({\underline{\alpha}}-1)+\left(n-\frac{5}{3}\right){\underline{\alpha}}=0$, or ${\underline{\alpha}}=\frac{n}{2}+\frac{1}{6}$.
		Hence,
		\[\hat{v}(\hat{\eta})\sim \hat{v}_{reg,+}(\hat{\eta}) +\hat{c}_{{\underline{\alpha}}}\hat{\eta}^{\frac{n}{2}+\frac{1}{6}}+\dots .\]
		{After possibly modifying the splitting $\hat{v} = \hat{v}_{reg,-}+\hat{v}_{sing,-}$, we may assume that $\hat{v}_{reg,+} = \hat{v}_{reg,-}$.} This concludes the proof of Lemma \ref{ATM3}.}

	\subsubsection*{Proof of Lemma \ref{ATM4}}
	The asymptotic expansion of the solution to \eqref{SelfSimilarEquation} takes the form $v(\eta)\sim\sum_{\alpha \in E}b_\alpha\eta^{\alpha}$ for some $E \subset \mathbb{C}$, and $b_{\alpha_0}\neq 0$ for some $\alpha_0$ of maximal real part. $\alpha_0$ is determined by substituting the expansion into \eqref{SelfSimilarEquation}: 
	\begin{equation}\label{inidicialeq}b_{\alpha_0}\textstyle{(-\alpha_0(\alpha_0-1) - \frac{2}{3}\alpha_0+\frac{2}{9})}\eta^{\alpha_0} -b_{\alpha_0}^{-2} \eta^{-2\alpha_0}  + \dots = 0,\end{equation}
	where dots indicate lower order terms. Depending on the sign of $\mathrm{Re}\ \alpha_0$, we obtain:\\
	When $\mathrm{Re}\ \alpha_0>0$, the nonlinear term $b_{\alpha_0}^{-2} \eta^{-2\alpha_0}$ is of lower order in \eqref{inidicialeq}, and $\alpha_0 = \frac{2}{3}$ is the unique positive root of $-\alpha_0(\alpha_0-1) - \frac{2}{3}\alpha_0+\frac{2}{9}=0$.   \\
	For $\mathrm{Re}\ \alpha_0 = 0$, \eqref{inidicialeq} leads to $\alpha_0 = 0$ and $b_0 = c^*$. Finally, for $\mathrm{Re}\ \alpha_0 < 0$, the nonlinear term is the leading term in \eqref{inidicialeq} and leads to the contradiction $b_{\alpha_0}^{-2}=0$.\\
	We conclude that either $\alpha_0=\frac{2}{3}$ or $\alpha_0 = 0$, and therefore the assertion of Lemma \ref{ATM4}.

	\subsection{Formal Proof of Global Existence of Self-Similar Solutions}

We now prove Theorem \ref{ATM5} formally.\\
	A smooth solution of \eqref{SelfSimilarEquation} is given by
	\begin{equation}\label{AsymSmoothSoln}
		v(\eta)\sim a\eta^{\frac{2}{3}}+ v_l(\eta).
	\end{equation}
	Here, $a$ is given in Proposition \ref{ATM1} and $v_l$ is a general solution to the linearized equation \eqref{AymLinearizedEqn}
	\be\label{AymLinearizedEqn}
	\textstyle{(1-\eta^2)\frac{d^2v_l}{d\eta^2}+\left(\frac{n-1}{\eta}-\frac{2}{3}\eta\right)\frac{dv_l}{d\eta}+\left(\frac{2}{9}+\frac{2}{v_s^3}\right)v_l=0,}
	\ee
which is obtained by linearizing \eqref{SelfSimilarEquation} around $v_s=a\eta^{\frac{2}{3}}$. Then $v_l$ can be found by first making a change of
variable, as earlier, {$s = \ln \eta$},  so that \eqref{AymLinearizedEqn} becomes, {using the definition of $v_s$,} 
	\begin{equation}\label{AymLinearizedEqns1}
\left( 1 - e^{2s} \right) \frac{d^2 v_l}{d s^2} +  \left( (n - 2) + \frac 13 e^{2s} \right) 
\frac{d v_l}{d s} + \left( \frac 43 \left( n - \frac 43 \right) + \frac 29 e^{2s} \right)
v_l = 0 ,
	\end{equation}
with $v_l$ regular at $s = 0$ {($\eta = 1$)}. We are interested in the asymptotic behaviour
of $v_l$ for $s \to - \infty$ ($\eta \to 0$), in particular for cases of $2 \le n \le 7$.

Note that the characteristic equation 
\[
\textstyle{\lambda^2 + (n - 2) \lambda + \frac 43 \left( n - \frac 43 \right) =0 .}
\]
associated with the limiting form of \eqref{AymLinearizedEqns1},
\[
\frac{d^2 v_l}{d s^2} +   \frac 13 e^{2s} \frac{d v_l}{d s} +  \frac 43 \left( n - \frac 43 \right) v_l = 0 ,
\]
got by taking $e^{2s}$ is small, {\it i.e.} $s$ large and negative,
has roots $(1 - \frac n2 ) \pm ib$, with $b = \sqrt{ \frac{n^2}4 - \frac 73 + 
\frac {25}9 }$. Then repeated use of variation of parameters to solve \eqref{AymLinearizedEqns1}, writing
\[
v_l = v_1 e^{(1 - \frac n2 )s} \cos bs + v_2 e^{(1 - \frac n2 )s} \sin bs \, ,
\]
gives for $2 \le n \le 7$, 
first that $v_1$ and $v_2$ are bounded for $s \to - \infty$, and then that
$v_1$ and $v_2$ both have limits as $s \to - \infty$. 

It follows that 
\[
v_l \sim B_1 e^{(1 - \frac n2 )s} \cos (bs + B_2) \, \mbox{ for } \, s \to - \infty
\]
and then that
	\begin{equation}\label{Asymv3}
		v \sim a\eta^{\frac{2}{3}}
+ B_1 \eta^{1 - \frac n2} \cos (b \ln \eta + B_2) \, \mbox{ for } \, \eta \to 0 \, ,
	\end{equation}
where $B_2$ is a constant determined by $n$. \footnote{It should be noted
that the ODE \eqref{AymLinearizedEqn} can be solved in terms of
hypergeometric functions, and the limiting behaviour \eqref{Asymv3}
got from this explicit solution. For example, in the case of $n=3$, we have
$v_l = C_l \eta^{1 - \frac n2} \mathrm{Im} \{ \beta_2 e^{i b \ln \eta} 
{_2\mathbf{F}_1} \left(\left[ \frac 23 - \frac n4 + \frac{i b}2, \frac 16 - \frac n4 + \frac{i b}2 \right], \left[ 1 + ib \right], \eta^2 \right)  \} $, $b$ as above with $n=3$, $\beta_2 = 
\left[ 6\Gamma \left( 1 - \frac{\sqrt{71}}6 i \right) \pi\sqrt{3} (-2)^{\frac 23} \right] \left/ \left[ 5\Gamma \left( \frac 23 \right) \Gamma \left( -\frac 1{12} - \frac{\sqrt{71}}{12}i \right) \Gamma \left( -\frac 7{12} - \frac{\sqrt{71}}{12} i \right) \right] \right.$ and $C_l$
is a real constant.}	

Matching \eqref{Asymv3} with \eqref{AsymForm1},
\[
v \sim a\eta^{\frac{2}{3}}+A_1c^{\frac{3}{4}n-\frac{1}{2}}\eta^{1-\frac{n}{2}}\cos\left(b\ln(\eta c^{-\frac{3}{2}})+A_2\right) 
= a\eta^{\frac{2}{3}} + A_1 c^{\frac{3}{4}n-\frac{1}{2}} \eta^{1-\frac{n}{2}}
\cos \left( b\ln\eta - \frac 32 b \ln c +A_2  \right) ,
\]
gives
	\begin{equation}\label{Matching4}
B_l =  A_1 c^{\frac{3}{4}n-\frac{1}{2}} , \quad
 B_2 + 2\pi k =  - \frac 32 b \ln c + A_2   ,
	\end{equation}
and
\begin{equation}\label{Matching4a}
- \ln c = \frac 2{3 b}(B_2 - A_2 + 2\pi k )
\end{equation}
where $k$ is a large positive integer.

	This results in a sequence of solutions $c_k$ with $c_k\geq0$.  The discrete set of such values $c_k$ for $c$ is what we look for.
	
	As we consider the smooth solution $v$ existing beyond $\eta=1$,  the integral formulation \eqref{IntBeyond1} in the proof of Lemma \ref{ExistenceCross1-G} leads to $\beta=0$. 
	
	Moreover, $v$ is monotonically increasing for all $\eta>0$, as we shall now demonstrate. 
	
	First, from the equation \eqref{SelfSimilarEquation}, it is easy to see $v$ is  monotonically increasing for  $\eta$ near $0$.  
	
	According to the formal argument above, with the asymptotic form of smooth $v$ for the  discrete  set  of  small  values $c_k\ll1$ of $c$,  $v$ remains 
close to $v^* = a \eta^{\frac{2}{3}}$ in bounded intervals such as $0 \leq \eta \leq \eta^* + \epsilon$, where $\eta^*$ is where $v^*$ crosses $c^*$: $a (\eta^*)^{\frac{2}{3}} = c^*$, $\eta^*>1$ for $2 \le n \le 7$, $\epsilon>0$ is small. 
	
Next note that, from the equation \eqref{SelfSimilarEquation}, $v$ has no local: \\
(a) maximum if $v < c^*$ and $0 < \eta < 1$; \quad (b) minimum if $v < c^*$ and $\eta > 1$; \\
(c) minimum if $v > c^*$ and $0 < \eta < 1$; \quad (d) maximum if $v > c^*$ and $\eta > 1$. \\

Now, for $c_k$ small enough (sufficiently large $k$), since $v$ is close to $v^*$,
it too will cross $c$ for the first time at some point $\eta^{**} > 1$: 
$v < c^*$ for $0 \le \eta < \eta^{**}$, $v = c^*$ at $\eta = \eta^{**}$.

In the interval $0 \le \eta \le 1$, $v < c^*$, $0 = v(0) < v(1)$, and by (a) above, there is no local maximum: Here $v$ is monotonic increasing.

In the interval $1 \le \eta < \eta^{**}$, $v < c^*$, $v(1) < v(\eta^{**})$, and by (b) above, there is no local minimum: Here $v$ is monotonic increasing.

Since $\eta^{**}$ is the first point at which $c^*$ is crossed, $v' (\eta^{**}) \ge 0$. By uniqueness of (regular) solutions of \eqref{SelfSimilarEquation}, $v' (\eta^{**}) \ne 0$, giving $v' (\eta^{**}) > 0$ and $v > c^*$ in a right neighbourhood of $\eta^{**}$. From (d) it follows that $v$ remains increasing for $\eta \ge \eta^{**}$.

	Hence, we conclude $v$ is increasing with respect to $\eta>0$ and $v>0$ for all $\eta>0$, as a result, the smooth solution $v$ exists globally and then, 
by Lemma~\ref{ATM4}, $v \sim b_{\frac 23} \eta^{\frac 23}$ for 
$\eta \to \infty$, with the value of $\sim b_{\frac 23}$ fixed by that of $c_k$.
	
	For $n\geq 8$, we don't know whether there is a smooth solution of the equation \eqref{SelfSimilarEquation}, according to the above analysis. 
	
	This concludes the formal proof of Theorem \ref{ATM5}.

\
\

Note that it \underline{might}, with some effort, be possible to make some of the arguments in the above proof rigorous by adapting arguments of \cite{collot2016stability} to otain error bounds for neglected terms in the formal asymptotic series. However, it should be noted that even if such work were to be successful, the results will be
expected to be weaker than those of \cite{collot2016stability} through the already noted propertiy of the key ODE \eqref{SelfSimilarEquation}: a global solution of the ODE cannot be expected for all initial values $c$; it is then not clear that the index $k$ will be identified with the number of times that solution $v_k$ crosses $v^*$.

	\section{Approximate Calculations of the Quenching Profile\label{sectionasymp}}
	
{	Having looked for self-similar solutions, which might serve to give the local form of quenching for solutions of \eqref{semi-linear-wave-equation} in the previous section, we now look at approximate solutions local to a quenching point. The relevance of the global self-similar solutions $v_j$ in Theorem \ref{ATM5} in \eqref{intro-5} hinges on their stability. }

{Should $u$ be only approximately self-similar, so that $v$ varies with a transformed time variable $\tau=-\ln(T_{\max}-t)$ as well as $\eta$, it is given by the PDE \eqref{SimialrityPDE} instead of the ODE \eqref{SelfSimilarEquation}: 
\begin{equation}\label{SimialrityPDE}
\frac{\partial^2v}{\partial\tau^2}+2\eta\frac{\partial^2v}{\partial\eta \, \partial\tau}-\frac{1}{3}\frac{\partial v}{\partial\tau}=(1-\eta^2)\frac{\partial^2v}{d\eta^2}+\left(\frac{n-1}{\eta}-\frac{2}{3}\eta\right)\frac{\partial v}{\partial\eta}+\frac{2}{9}v-\frac{1}{v^2} ;
\end{equation}
note that the new variable satisfies $\tau \to \infty$ as $t \to T_{\max}$.
To discuss stability of a true self-similar solution $v_j$, we interpret it as an equilibrium of the PDE \eqref{SimialrityPDE} for $v(\eta, \tau)$.}

{In this section go into this by showing formal and numerical evidence for the following conjecture in $n$ dimensions:
\begin{conjecture}\label{mainconj}
a) The trivial self-similar solution $v_0 = c^*$ to \eqref{SelfSimilarEquation} is stable as a solution of the PDE \eqref{SimialrityPDE}, excepting for (spatially uniform) perturbations which grow as exp $\tau$, which correspond to solutions quenching at different times.\\
b) The generic local behaviour of a {strictly increasing} radial quenching solution to \eqref{semi-linear-wave-equation-R} near the quenching time $T_{\max}$ is given by $u(r,t) = \left(T_{\max}-t\right)^{\frac{2}{3}}V(\xi,t)$, with $\xi=\frac{r}{(T_{\max}-t)^{1/2}}$ and
\begin{align}\label{QuenchingProfile}
	V(\xi,t) \sim C\xi^{\frac{4}{3}} \quad \text{{as $\xi \to \infty$} and } \quad V(0,t) \to c^*
\end{align}
for $t \to T_{\max}$. Correspondingly,
\begin{equation}
u(r,T_{\max}) \sim C r^{\frac{4}{3}} \quad \text{as $r \to 0$ \quad with }\quad
u(0,t) \sim c^* (T_{\max}- t)^{\frac 23} \text{ as t $\to T_{\max}$,}
\end{equation}
{provided $u$ is a strictly increasing function of $r$ near the quenching time.} Here $C$ is a constant which depends on $n$, and on the initial and boundary conditions for the semi-linear wave equation \eqref{semi-linear-wave-equation-R}. \\
c) For $2 \le n \le 7$, the smooth self-similar solutions $v_j$ to \eqref{SelfSimilarEquation} are unstable for $j \in \mathbb{N}$ as solutions of the PDE \eqref{SimialrityPDE}.
\end{conjecture}
The behaviour in c) is supported by the formal and numerical results in the paper \cite{TNLK} for the problem in one space dimension. }

We start in Subsec.~\ref{subsec:numerics} by looking at numerical solutions of the PDE \eqref{semi-linear-wave-equation-R} for dimension $n=2$. In the Subsecs.~\ref{SS:inner} - \ref{SS:matching} we consider a formal asymptotic solution of the PDE in any dimension $n$.
	
	\subsection{Numerical Solution of the PDE}\label{subsec:numerics}

\begin{figure}[t]
                \begin{center}
                        {\includegraphics[width=0.4\textwidth]{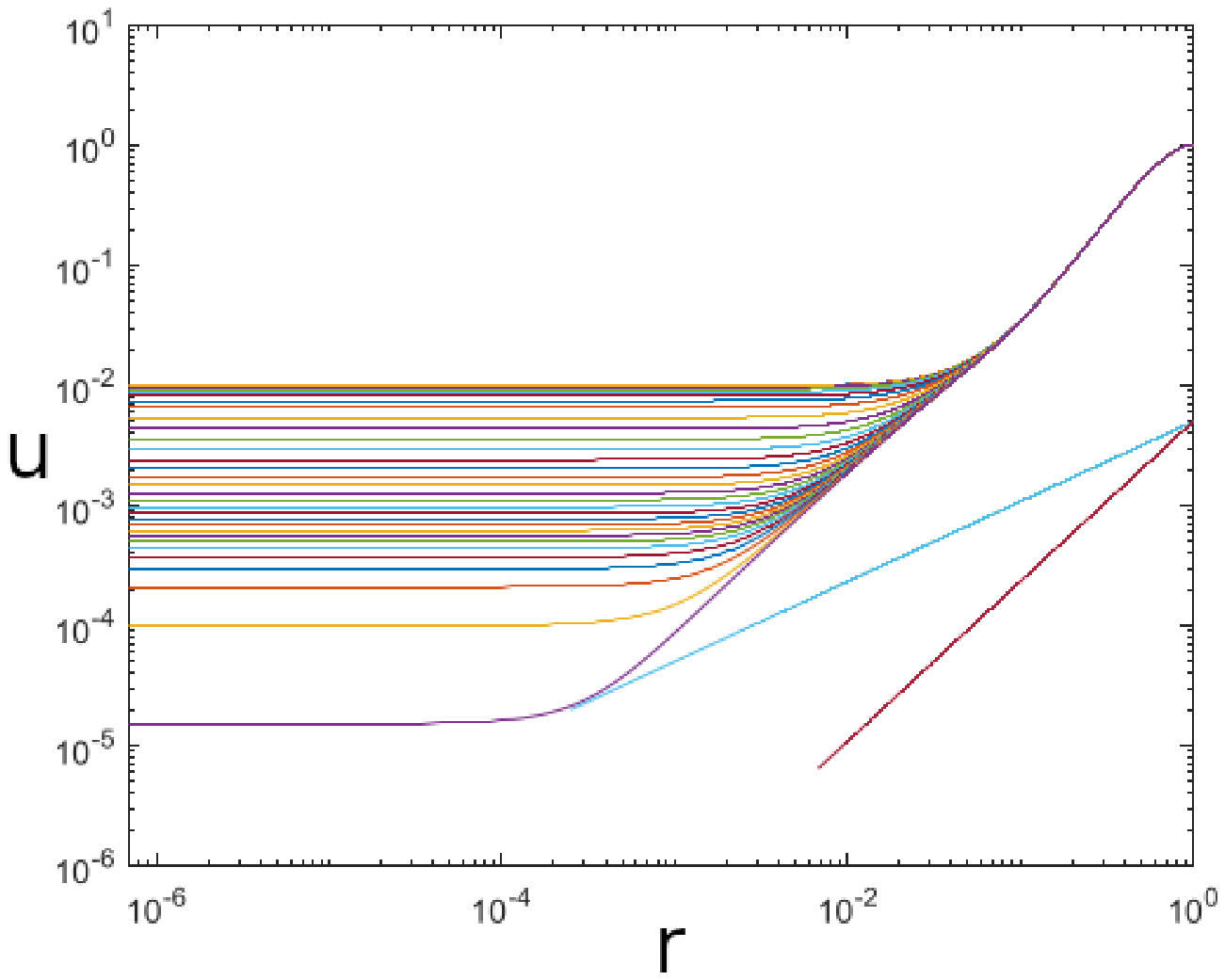}}\hspace*{0.55cm}
                        {\includegraphics[width=0.39\textwidth]{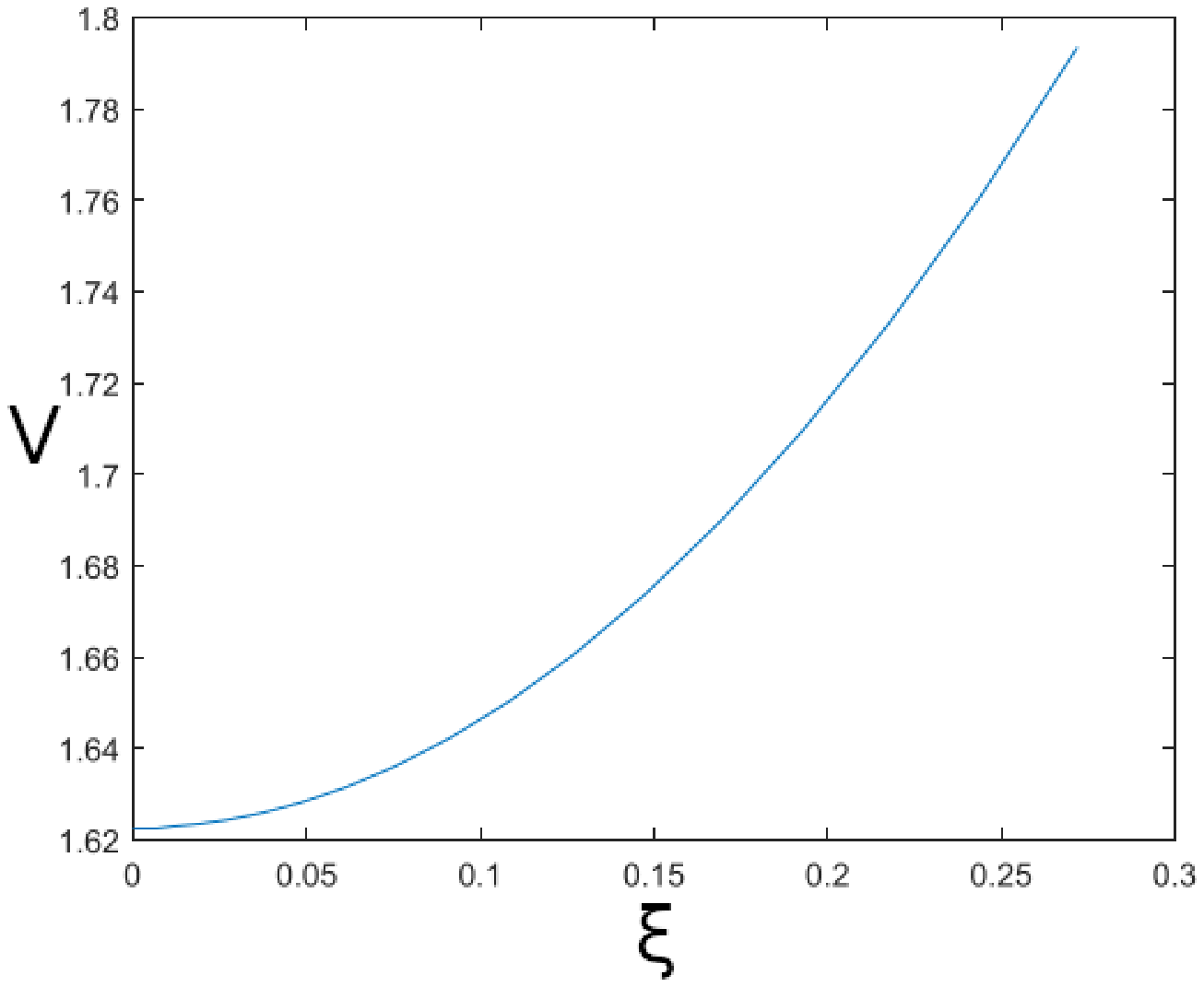}}
                        \caption{(A) Solution $u$ for $a=1$, (B) rescaled solution $V$ for small $\xi$.}
                        \label{fig1}
                \end{center}
        \end{figure}

        \begin{figure}[t]
                \begin{center}
                        {\includegraphics[width=0.46\textwidth]{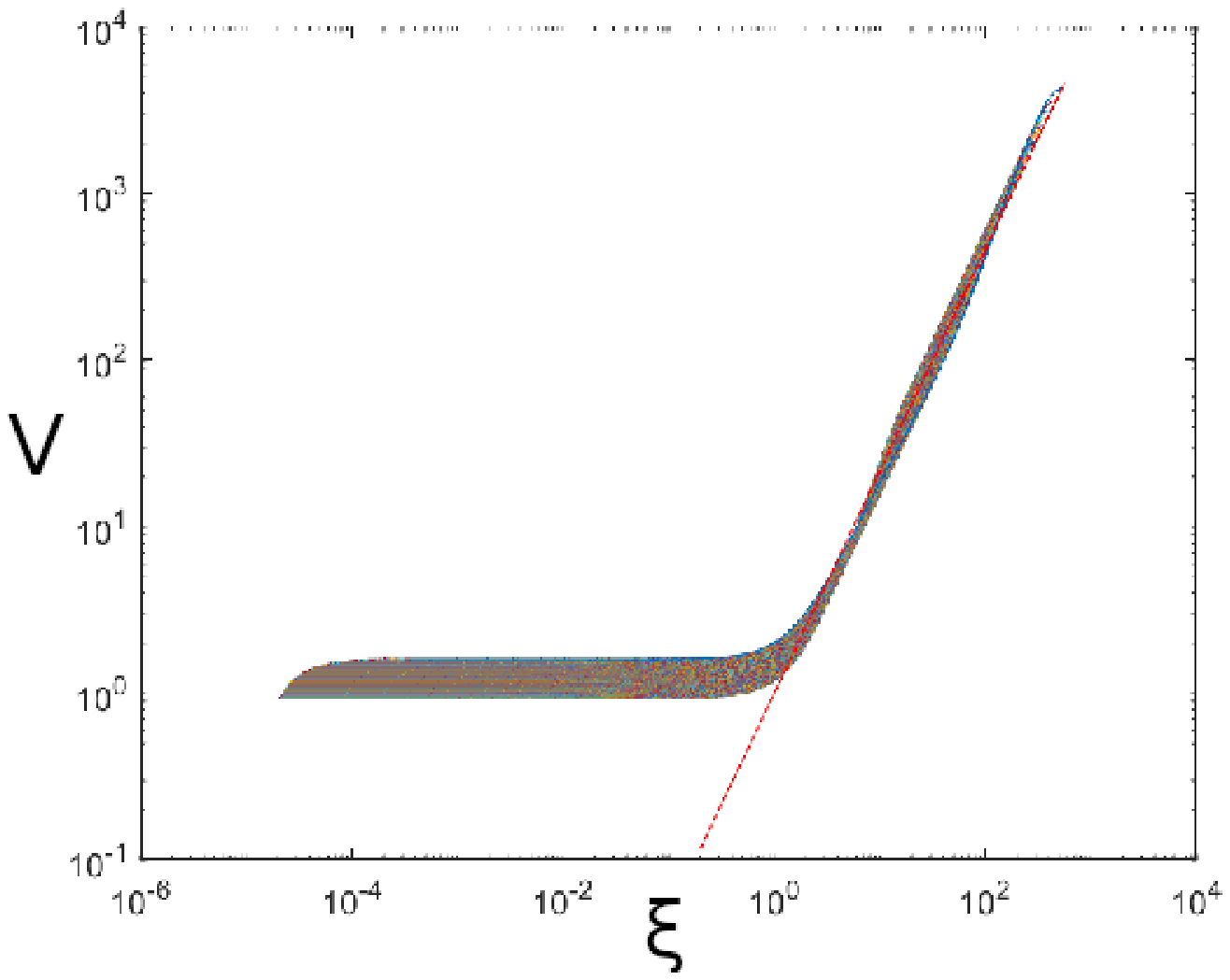}}
                        {\includegraphics[width=0.46\textwidth]{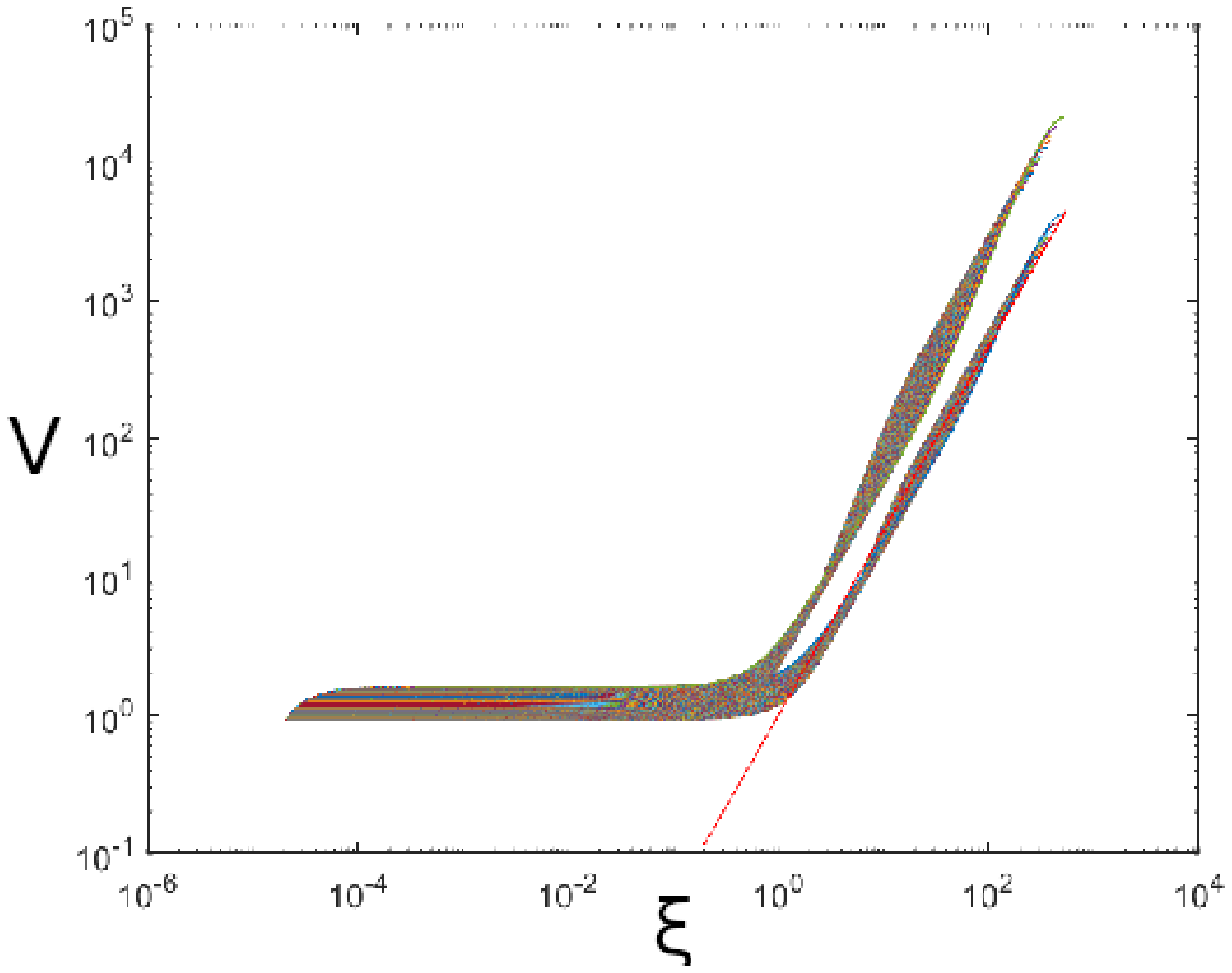}}
                        \caption{Rescaled solution $V$ for (A) $a=1$, (B) $a=1$ and $a=5$.}
                        \label{fig2}
                \end{center}
        \end{figure}

	We illustrate the behaviour of radial solutions near the quenching time with numerical experiments. We consider the semi-linear wave equation \eqref{semi-linear-wave-equation} in the unit ball $\Omega = \{x \in \mathbb{R}^2 : |x|<1\}$ with homogeneous Neumann boundary conditions $\partial_\nu u = 0$ on $\partial \Omega$.  The initial condition is of the form
	$$u(x,0) = \textstyle{\epsilon+a \sin\left(\frac{\pi |x|}{2}\right)^2,\quad \frac{\partial u}{\partial t}(x,0) =0\qquad (x \in \Omega)} ,$$
	where $\epsilon = 10^{-2}$ and $a\in \{1,5\}$. {As we are looking at the local behaviour near the quenching point,}  the  following observations are independent of the type of boundary condition and the precise boundary data. They seem to reflect the behaviour of the solution not only in this special case, but for generic initial conditions in the numerically feasible dimensions $n=2, 3,4,5$. The results refine and complement those obtained in  \cite{TNLK} for dimension $n=1$. 

We discretise the weak formulation of \eqref{semi-linear-wave-equation} by piecewise linear finite elements in $r \in (0,1)$ on an algebraically graded mesh with nodes $r_j = \left(\frac{j}{N}\right)^2$, $0\leq j \leq N$, $N$ fixed. Note that integration in polar coordinates involves a Jacobian $r^{n-1}$, with resulting numerical difficulties for $n \geq 6$. An implicit Euler method with variable time step is used for the time discretisation. Note that the numerical method becomes ill-conditioned with increasing dimension $n$.
	
	For $a=1$, Figure \ref{fig1}(A) shows $u$ as a function of the radius $r=|x|$ at several times $t$  close to $T_{\max}$. Plotting both $u$ and $r$ in a logarithmic scale, a solution $u \sim \widetilde{C} r^{\alpha}$ corresponds to a straight line of slope $\alpha$. The straight red line corresponds to $\alpha = \frac{4}{3}$, in agreement with the behaviour of $u$ {at all times $t$}. The straight blue line depicts the behaviour expected for non-trivial self-similar solutions, where $\alpha = \frac{2}{3}$.
	
	Figure \ref{fig2}(A) depicts $V =  (T_{\max}-t)^{-2/3} u(r,t)$ as a function of the rescaled radius $\xi = \frac{r}{(T_{\max}-t)^{1/2}}$, again in logarithmic scales. As $t \to T_{\max}$, $V$ approaches a limiting profile, indicating an asymptotically self-similar evolution. The straight red line agrees with $V(\xi,t) \sim C \xi^{4/3}$ as $t \to T_{\max}$. Note the scaling of the variable $\xi$ in agreement with Conjecture \ref{mainconj}, rather than \eqref{SelfSimilarVariables}.
	
	Figure \ref{fig2}(B) compares the rescaled solutions $V =  (T_{\max}-t)^{-2/3} u(r,t)$ for initial data corresponding to $a=1$, respectively $a=5$. In both cases $V(\xi,t) \sim C \xi^{4/3}$ as $t \to T_{\max}$, but the constant $C$ depends on the initial condition. 
	
	Figure \ref{fig1}(B) plots $V=  (T_{\max}-t)^{-2/3} u(r,t)$ near $r = 0$, when $a=1$ and $t$ is the last time step before $T_{\max}$. For $t \to T_{\max}$ and $\xi \to 0$, $V$ is independent of $t$, and the limiting value $V(0,t)\simeq 1.62$ is compatible with $c^*=\left(\frac{9}{2}\right)^{\frac{1}{3}}\approx1.65$.

	\subsection{Formal Asymptotics: Inner Solution\label{SS:inner}} 
	
	``Inner solution" refers to an asymptotic form of a quasi-self-similar solution $v$ satisfying \eqref{SimialrityPDE} for $\eta$ not large.  We assume that $v$ has an asymptotic form 
	\[v\sim c^*+\underbar{V},\quad \text{with}\quad \lim_{\tau\to\infty}\underbar{V}(\eta, \tau)=0.\]
	We look for an asymptotic form of $\underbar{V}$, $\underbar{V}\sim e^{-\lambda\tau}\varphi_\lambda(\eta)$, where $\varphi_\lambda$ satisfies
	\begin{equation}\label{QuenchingEq}
		(1-\eta^2)\frac{d^2\varphi}{d\eta^2}+\left[2\left(\lambda-\frac{1}{3}\right)\eta+\frac{n-1}{\eta}\right]\frac{d\varphi_\lambda}{d\eta}+\left(\frac{2}{3}-\frac{\lambda}{3}-\lambda^2\right)\varphi_\lambda=0.
	\end{equation}
	We seek an even solution of  \eqref{QuenchingEq} as a power-series given by 
	\begin{equation}\label{PSSoln}
		\varphi_\lambda(\eta)=\textstyle{\sum_{k=0}^{\infty} a_{k}\eta^{2k}.}
	\end{equation}
	Plugging \eqref{PSSoln} into \eqref{QuenchingEq}, after some manipulations and ratio test, we see that the radius of convergence of the power series for $\varphi_\lambda(\eta)$ is $1$, consistent with $\varphi_\lambda$ having a singularity at $\eta=1$,  the solution \eqref{PSSoln} is not smooth unless it terminates. The series \eqref{PSSoln} terminates, with $\varphi_\lambda$ being a polynomial, and hence smooth for all $\eta$,  for $\lambda=\lambda_k^{\pm}$ with $\lambda_k^{-}=2k-1$ for $k=0,\ 1,\ 2$,\dots and $\lambda_k^{+}=2k+\frac{2}{3}$ for $k=0,\ 1,\ 2$,\dots \\
	The case $k=0$ and $\lambda=-1$ corresponds to a shift in $T_{\max}$, which we ignore.\\
	For $k=0$ and $\lambda=\frac{2}{3}$, $\underbar{V}$ is a slower shrinking solution, uniform in $\eta$, and the asymptotic form of $\underbar{V}$ is given by
	\[\underbar{V}\sim be^{-\frac{2}{3}\tau},\quad \text{where}\ b\ \text{is an arbitrary constant}.\]
	For $k=1$ and $\lambda=1$, the asymptotic form of $\underbar{V}$ is quadratic in $\eta$,
	\[\underbar{V}\sim ae^{-\tau}\left(\eta^2+3n\right),\quad \text{where}\ a\ \text{is an arbitrary constant}.\]
	We consider these two terms, as they are the slowest decaying in time, and therefore dominant. Other contributions decay more rapidly and we neglect them. As a result, 
	\begin{equation}\label{AsymPSSoln}
		v\sim c^*+be^{-\frac{2}{3}\tau}+ae^{-\tau}\left(\eta^2+3n\right)+\dots.
	\end{equation}
	Returning to earlier variables \eqref{SelfSimilarVariables}, for $t\to T_{\max}$, $\eta\to\infty$,  we obtain
	\begin{align}
		u\sim& c^*\left(T_{\max}-t\right)^{\frac{2}{3}}+b\left(T_{\max}-t\right)^{\frac{4}{3}}+a\left(T_{\max}-t\right)^{\frac{5}{3}}\left(\eta^2+3n\right)+\dots \notag\\
		\sim&\left(T_{\max}-t\right)^{\frac{2}{3}}\left(c^*+b\left(T_{\max}-t\right)^{\frac{2}{3}}+a\left(T_{\max}-t\right)\eta^2+\dots\right)\notag\\
		\sim&\left(T_{\max}-t\right)^{\frac{2}{3}}\left(c^*+b\left(T_{\max}-t\right)^{\frac{2}{3}}+a\left(T_{\max}-t\right)^{-1}r^2+\dots\right),\label{InnerSoln}
	\end{align}
	for $t\to T_{\max}$, $r\to 0$. Here terms $c^*$ and $a\left(T_{\max}-t\right)^{-1}r^2$ are in the same size for $\eta\sim O\left[\left(T_{\max}-t\right)^{-\frac{1}{2}}\right]$, $\eta\to \infty$ i.e. $r\sim O\left[\left(T_{\max}-t\right)^{\frac{1}{2}}\right]$, $r\to 0$.
	
	This indicates that there is significant behaviour of the solution $v$ of \eqref{SimialrityPDE} and $u$ for large $\eta$ of the size $\left(T_{\max}-t\right)^{-\frac{1}{2}}$.
	
	\subsection{Formal Asymptotics: Outer Solution\label{SS:outer}}
	
	We want an approximation for a solution of equation \eqref{SimialrityPDE} valid for large $\eta$. 
	
	For large $\eta\gg1$, terms $\frac{\partial^2v}{\partial\eta^2}$ and $\frac{n-1}{\eta}\frac{\partial v}{\partial\eta}$ in \eqref{SimialrityPDE} are relatively small, i.e. we might neglect the Laplacian term $\Delta u$ in the original equation \eqref{semi-linear-wave-equation} to get 
	\begin{equation}\label{OuterEq}
		\textstyle{\frac{d^2u}{dt^2}=-\frac{1}{u^2}}
	\end{equation}
	which has the exact solution given by (see (6.7) in \cite{TNLK})
	\begin{equation}\label{OuterSoln}
		\textstyle{T_{\max}-t+A=\frac{u^{\frac{1}{2}}}{B}\left(2+Bu\right)^{\frac{1}{2}}-2B^{-\frac{3}{2}}\ln\left[\left(\frac{Bu}{2}\right)^{\frac{1}{2}}+\left(1+\frac{Bu}{2}\right)^{\frac{1}{2}}\right],}
	\end{equation}
	where $A=A(r)>0$, $B=B(r)>0$ with $\displaystyle\lim_{r\to 0}A(r)=\lim_{r\to 0}B(r)=0$.
	
	Taking $u\to 0$ in \eqref{OuterSoln} and doing the appropriate eliminations give 
	\begin{equation}\label{OuterSoln2}
		\textstyle{T_{\max}-t+A(r)\sim \frac{\sqrt{2}}{3}u^{\frac{3}{2}}+c_o B(r)u^{\frac{5}{2}}+\dots}
	\end{equation}
	Here and the following parts,  $c_o$ denotes a number, no physical dependence. This implies 
	\[	\textstyle{u\sim \left[\frac{3}{\sqrt{2}} (T_{\max}-t)+\dots\right]^{\frac{2}{3}}.}\]
	In more detail
	\begin{align}
		u\sim&\left[\frac{3}{\sqrt{2}} (T_{\max}-t)+A(r)+c_oB(r)(T_{\max}-t)^{\frac{5}{3}}+\dots\right]^{\frac{2}{3}}\notag\\
		\sim&c^*(T_{\max}-t)^{\frac{2}{3}}\left[1+c_oA(r)(T_{\max}-t)^{-1}+c_oB(r)(T_{\max}-t)^{\frac{2}{3}}+\dots\right]^{\frac{2}{3}}\notag\\
		\sim&(T_{\max}-t)^{\frac{2}{3}}\left[c^*+c_oA(r)(T_{\max}-t)^{-1}+c_oB(r)(T_{\max}-t)^{\frac{2}{3}}+\dots\right].\label{OuterSoln3}
	\end{align}
	Another way for solving the equation \eqref{OuterEq} can be seen in \cite{TNLK}.
	
	\subsection{Formal Asymptotics: Matching\label{SS:matching}}
	
	Recall the approximation for the inner region is of the form \eqref{InnerSoln} and the the approximation for the outer region has the form \eqref{OuterSoln3}. In an intermediate region, the inner solution \eqref{InnerSoln} and the approximation \eqref{OuterSoln3} to outer solution should be the same, which implies 
	\[B(r)\sim c_ob\quad \text{for}\ r\to 0,\quad A(r)\sim c_oar^2\quad \text{for}\ r\to 0.\]
	\subsection{Quenching Profile} Taking $t\to T_{\max}$ in \eqref{OuterSoln} gives 
	\[A=\textstyle{\frac{u^{\frac{1}{2}}}{B}\left(2+Bu\right)^{\frac{1}{2}}-2B^{-\frac{3}{2}}\ln\left[\left(\frac{Bu}{2}\right)^{\frac{1}{2}}+\left(1+\frac{Bu}{2}\right)^{\frac{1}{2}}\right].}\]
	Then for $r\to 0$ and $u\to 0$, we obtain 
	\[c_oar^2\sim \textstyle{\frac{\sqrt{2}}{3}u^{\frac{3}{2}}}\]
	as \eqref{OuterSoln2}. Thus, the local profile of the solution at quenching time $t=T_{\max}$ is 
	\[u\sim c_oa^{\frac{2}{3}}r^{\frac{4}{3}}.\]
	The coefficient $a$ is  determined by the initial and boundary conditions of the
	equation \eqref{semi-linear-wave-equation}.
	
	This gives us an $r^{\frac{4}{3}}$ dependence of the solution profile near the quenching point $r=0$, which agrees with the numerical results in Section \ref{subsec:numerics}.

	
	\section{Discussion\label{Disc}}


Motivated by touch-down singularities in MEMS capacitors, we have, in this article, studied the (radial) profile of solutions to the semi-linear wave equation \eqref{semi-linear-wave-equation} near a quenching time. The results are related to classical works on self-similar solutions for the stationary problem \cite{gel1959some, JL}. They have led to a detailed conjecture for quenching behaviour in \eqref{semi-linear-wave-equation}, with a different scaling from that in recent rigorous work on blow-up \cite{donninger2012stables, merle2019blow}.
	
	In Section \ref{GloExist} we found that there were non-trivial smooth self-similar solutions which could, potentially, give the local profile of quenching solutions to the semi-linear wave equation \eqref{semi-linear-wave-equation} in dimension $n$ from 2 to 7, the dimension $n=2$ being physically relevant for touchdown in MEMS devices. Quenching governed by such a self-similar solution would produce a local quenching profile $u(r,T_{\max}) \sim C_{nm} r^{\frac23}$ for $r \to 0$, with the positive constant $C_{nm}$ fixed by dimension $n$ and index $m$ specifying the particular similarity solution, equivalent to a specific value of the shooting parameter $c=v(0)$ in Section~\ref{GloExist}.
	
	However, in Section \ref{sectionasymp} we found that both numerical solutions and formal asymptotics indicate that these non-trivial self-similar solutions ($C_{nm} > 0$) are unstable, while the trivial, constant solution is stable: generically, the solution to \eqref{SimialrityPDE} satisfies $v\to c^*=\left(\frac{9}{2}\right)^{\frac{1}{3}}$ as $\tau \to \infty$ for fixed $r$. This then means that, at least for radially symmetric cases, as quenching is approached, i.e.~$t \to T_{\max}$, $u(0,t) \sim c^* t^{\frac23}$, and that at quenching time, $t = T_{\max}$, $u$ has a flatter local profile, $u(r,T_{\max}) \sim C r^{\frac43}$ for $r \to 0$, with the positive constant $C$ fixed by the imposed initial and boundary conditions. These results lend support to the detailed behaviour of the solution proposed in Conjecture \ref{mainconj} at the quenching time. Such less sharp local behaviour near a quenching point gives, at the quenching time, a nonlinear term in \eqref{semi-linear-wave-equation} of size $r^{- \frac 83}$, whose integral is infinite for $n=1$, 2.  This could suggest that, for these physically important cases, any continuation of the solution beyond quenching time might require significant modification of the model.
	
	\appendix
	\section{Local Existence of Self-Similar Solutions}\label{LocExist}
	
	{In this Appendix we prove} the local solvability of the initial value problem given by \eqref{SelfSimilarEquation}, \eqref{InitialValue0} {in any dimension $n$}. 
	\begin{proposition}\label{Existence of Self-Similar Solution for n>1}
		(A) Equation \eqref{SelfSimilarEquation} with the initial condition \eqref{InitialValue0} admits a solution $v$ with maximal existence interval $[0,\ \eta_{\max})$. 
		
		(B) If $\eta_{\max}<\infty$, then
		\[\textstyle{\lim_{\eta\rightarrow \eta_{\max}^-}}v(\eta)=0.\]

		(C) If $\eta_{max} \leq 1$, the solution $v$ is unique.
		
		(D) If $\eta_{max} > 1$, there exists a $1$-parameter family of solutions $\{v_\beta\}_{\beta \in \mathbb{R}}$ to \eqref{SelfSimilarEquation} with initial condition \eqref{InitialValue0}. $\eta_{\max}$ depends on both $v_\beta(0)$ and $\beta$. $v_\beta$ is independent of $\beta$ in the intervall $[0,1]$, and it admits a decomposition $v_\beta = v_{\beta,reg}+v_{\beta,sing}$ into a smooth  and a singular part at $\eta=1^+$. We have $$v_{\beta,sing}(\eta) =  \textstyle{\sum_{\alpha \in  E}c_\alpha (\eta-1)^\alpha}\ ,$$ whose coefficients $c_\alpha$ are determined by $\beta$, and the smallest exponent $\alpha \in E$ is given by $\alpha=\frac{n}{2}+\frac{1}{6}$.
	\end{proposition}

	\subsection{Integral Formulation of Self-Similar Equation}\label{LocExistSec1}
	For the proof of Pro\-position \ref{Existence of Self-Similar Solution for n>1}, we show that there exists a unique local solution to \eqref{SelfSimilarEquation}, with initial conditions imposed at $\eta_0\geq 0$ given by
	\be\label{InitialValue}
	v(\eta_0)=c>0,\quad \textstyle{\frac{dv}{d\eta}}(\eta_0)=d
	\ee
	{when $\eta_0 \neq 1$ (see Section \ref{LocExistSec4} for details when $\eta_0=1$).} The Picard-Lindel\"{o}f Theorem implies the existence of a unique local solution 
	in an intervall $[\eta_0, \eta_0+\delta)$ away from $\eta =0$ and $1$, where $\delta>0$. To prove Proposition \ref{Existence of Self-Similar Solution for n>1} it is then sufficient to consider local existence and uniqueness for $\eta_0$ near the singular points $\eta=0,1$. {The expansion in Proposition \ref{Existence of Self-Similar Solution for n>1}(D) follows from Lemma \ref{ATM3}.}
	
	We reformulate \eqref{SelfSimilarEquation}, \eqref{InitialValue} as an integral equation, which is obtained from the linearisation \eqref{SelfSimilarEquation} around $c^*$, given by
	\be\label{LinearizedSelfSimilarEquation-E}
	(1-\eta^2)\frac{d^2\bar{v}}{d\eta^2}+\left(\frac{n-1}{\eta}-\frac{2}{3}\eta\right)\frac{d\bar{v}}{d\eta}+\frac{2}{3}\bar{v}=0.
	\ee
	For $n>2$ the following hypergeometric functions $P$, $\overline{P}$ form a basis of solutions to the linearized equation \eqref{LinearizedSelfSimilarEquation-E}: 
	$$\textstyle{P(\eta)=\eta^{2-n}{_2\mathbf{F}_1}\left(\left[\frac{4}{3}-\frac{n}{2},\ \frac{1}{2}-\frac{n}{2}\right],\ \left[2-\frac{n}{2}\right],\ \eta^2\right), \ \overline{P}(\eta)={_2\mathbf{F}_1}\left(\left[-\frac{1}{2},\ \frac{1}{3}\right],\ \left[\frac{n}{2}\right],\ \eta^2\right).}$$
	$P$ is smooth near $\eta=1$, with $P(1)=1$, while $\overline{P}$ is smooth near $\eta=0$ and $\overline{P}(0)=1$. In fact, $P$ is a rational function for odd $n$ and, specifically, $P(\eta) = \frac{\eta^2+3}{4\eta}$ when $n=3$. A basis for $n=1$ is given by $1$ and $\eta$, while for $n=2$ a basis is given by
	$$\textstyle{{_2\mathbf{F}_1}\left(\left[-\frac{1}{2},\ \frac{1}{3}\right], \left[-\frac{1}{6}\right], -\eta^2 + 1\right), (\eta^2 - 1)^{7/6} {_2\mathbf{F}_1}\left(\left[\frac{2}{3},\ \frac{3}{2}\right], \left[\frac{13}{6}\right], -\eta^2 + 1\right)}\ .$$
	
	Given a solution $Q$ of the linearized equation \eqref{LinearizedSelfSimilarEquation-E}, with $Q(\eta_0)>0$, we set
	\be\label{Transformation2}
	v(\eta)=c^*+w(\eta, z),\quad w(\eta, z)= Q(\eta)z,\quad z=z(\eta)\ .
	\ee
	The self-similar equation \eqref{SelfSimilarEquation} then becomes
	\be\label{SelfSimilarEquation4}
	\frac{d^2z}{d\eta^2}+\left(\frac{2}{Q(\eta)}\frac{dQ(\eta)}{d\eta}+\frac{3(n-1)-2\eta^2}{3\eta(1-\eta^2)}\right)\frac{dz}{d\eta}=\frac{f(w(\eta, z))}{(1-\eta^2)Q(\eta)},
	\ee
	where
	\[f(w(\eta, z))=\frac{1}{[c^*+w(\eta, z)]^{2}}-\frac{1}{(c^*)^{2}}+\frac{2}{(c^*)^3}w(\eta, z)\ .\]
	The initial condition \eqref{InitialValue} becomes
	\be\label{InitialValue3}
	z(\eta_0)=\frac{c-c^*}{Q(\eta_0)}=:z_0,\quad \frac{dz}{d\eta}(\eta_0)=\frac{d}{Q(\eta_0)}-\frac{c-c^*}{Q^2(\eta_0)}\frac{dQ}{d\eta}(\eta_0)=:z_0^*\ .
	\ee
	With the substitution  $y=\frac{dz}{d\eta}$, one readily derives the explicit solution to \eqref{SelfSimilarEquation4}, \eqref{InitialValue3}:
	\begin{align}
		z(\eta)&=z_0+\alpha_1\int_{\eta_0}^{\min\{\eta,1\}}\frac{\left(\zeta^2-1\right)^{\frac{n}{2}-\frac{5}{6}}}{Q^2(\zeta)\zeta^{n-1}}d\zeta + \beta\int_{1}^{\max\{\eta,1\}}\frac{\left(\zeta^2-1\right)^{\frac{n}{2}-\frac{5}{6}}}{Q^2(\zeta)\zeta^{n-1}}d\zeta\notag\\ &
		\quad -\int_{\eta_0}^{\eta}\left[\int_{\xi}^{\eta}\frac{\left(\zeta^2-1\right)^{\frac{n}{2}-\frac{5}{6}}}{Q^2(\zeta)\zeta^{n-1}}d\zeta\right]\left[
		\frac{Q(\xi)\xi^{n-1}}{\left(\xi^2-1\right)^{\frac{1}{6}+\frac{n}{2}}}\right]f(w(\xi, z))d\xi \ \label{IntegralSelfSimilarEquation},
	\end{align}
	where $\alpha_1$ is determined by the initial condition at $\eta_0$,
	\begin{equation}\label{IntegralSelfSimilarEquationAlpha1}
		\alpha_1 = z_0^*Q^2(\eta_0)\left(\eta_0^2-1\right)^{\frac{5}{6}-\frac{n}{2}}\eta_0^{n-1}\ ,
	\end{equation}
	$\beta \in \mathbb{R}$ is an arbitrary constant, and $H$ denotes the Heaviside function. 
	Note that the existence of a unique solution to the differential equation \eqref{SelfSimilarEquation4}, \eqref{InitialValue3} is equivalent to the existence of a unique solution to the integral equation \eqref{IntegralSelfSimilarEquation}, \eqref{IntegralSelfSimilarEquationAlpha1}.
	
	%
	
	\subsection{Local Existence of Unique Self-Similar Solution for $\eta_0$ near $1^-$}\label{LocExistSec2}
	We consider $\eta_0 = 1-\epsilon$ for a sufficiently small $\epsilon\in(0,\varepsilon_0(n))$, $n>2$. The function $Q=P$ is smooth near $\eta=1$ and satisfies $P(1)=1$. After possibly shrinking $\varepsilon_0(n)$,
	\be\label{P1}
	\textstyle{\frac{1}{2}}\leq P(\eta)\leq 2,\ \ \text{for all}\ 1-\varepsilon_0(n)\leq\eta\leq1.
	\ee
	For the integral formulation given by \eqref{IntegralSelfSimilarEquation}, we prove the local existence of a solution, as stated in the following lemma:
	\begin{lemma}\label{ExistenceNear1-G}
		{Let $\eta_0 = 1-\epsilon$, $\epsilon\in(0,\varepsilon_0(n))$.} There exists a $\delta>0$, depending only on the initial values \eqref{InitialValue3} and on the dimension $n$, such that  the integral equation \eqref{IntegralSelfSimilarEquation} has a unique solution in the interval $\mathcal{I}_\delta = \left[\eta_0,\ \min\{\eta_0+\delta,\ 1\}\right]$. 
	\end{lemma}
	\begin{proof}
		For given $\delta>0$, to be fixed later, we set $\mathcal{I}_\delta = \left[\eta_0,\ \min\{\eta_0+\delta,\ 1\}\right]$ and
		\[B_\delta=\left\{z\in C\left(\mathcal{I}_\delta; \mathds{R}\right):\ z(\eta_0)=z_0,\ \sigma_1\leq c^*+w(\eta, z)\leq\sigma_2\right\}.\]
		Here, $\sigma_1$ and $\sigma_2$ are positive constants depending on the initial values \eqref{InitialValue3}. \\Consider the nonlinear operator $\Gamma$ on $B_\delta$ given by \eqref{IntegralSelfSimilarEquation},
		\[
		\textstyle{\Gamma z(\eta)= z_0+\alpha_1\int_{\eta_0}^{\eta}\frac{\left(\zeta^2-1\right)^{\frac{n}{2}-\frac{5}{6}}}{P^2(\zeta)\zeta^{n-1}}d\zeta
			-\int_{\eta_0}^{\eta}\left[\int_{\xi}^{\eta}\frac{\left(\zeta^2-1\right)^{\frac{n}{2}-\frac{5}{6}}}{P^2(\zeta)\zeta^{n-1}}d\zeta\right]
			\frac{P(\xi)\xi^{n-1}}{\left(\xi^2-1\right)^{\frac{1}{6}+\frac{n}{2}}}f(w(\xi, z))d\xi.}
		\]
		
		\noindent\textbf{Contraction.} We first prove that $\Gamma$ is a contractive map on $B_{\delta_1}$, for $\delta_1>0$ sufficiently small. For $z_1, z_2 \in B_{\delta_1}$, let $w_1=w(\xi, z_1)$, $w_2=w(\xi, z_2)$. Then
		\begin{align}
			\left|f(w_1)-f(w_2)\right|=\textstyle{\left|\frac{1}{(c^*+w_1)^2}-\frac{1}{(c^*+w_2)^2}+\frac{2}{(c^*)^3}(w_1-w_2)\right|}\qquad \qquad\qquad \quad \label{GEst-0}\\
			\leq\textstyle{\left|\frac{1}{(c^*+w_1)^2(c^*+w_2)}+\frac{1}{(c^*+w_1)(c^*+w_2)^2}+\frac{2}{(c^*)^3}\right|\left|w_1-w_2\right|
				\leq\left|\frac{2}{(\sigma_1)^3}+\frac{2}{(c^*)^3}\right|\left|w_1-w_2\right|}.\notag
		\end{align}
		Hence, 
		\begin{align}
			\left|\Gamma z_1(\eta)-\Gamma z_2(\eta)\right|
			=\textstyle{\left|\int_{\eta_0}^{\eta}\left[\int_{\xi}^{\eta}\frac{\left(\zeta^2-1\right)^{\frac{n}{2}-\frac{5}{6}}}{P^2(\zeta)\zeta^{n-1}}d\zeta\right]
				\left[\frac{P(\xi)\xi^{n-1}}{\left(\xi^2-1\right)^{\frac{1}{6}+\frac{n}{2}}}\right][f(w_1)-f(w_2)]d\xi\right|}\notag\\
			\leq\textstyle{\left|\frac{2}{(\sigma_1)^3}+\frac{2}{(c^*)^3}\right|} \|w_1-w_2\|_{C(\mathcal{I}_\delta; \mathds{R})} \int_{\eta_0}^{\eta}\left|\int_{\xi}^{\eta}\frac{\left(\zeta^2-1\right)^{\frac{n}{2}-\frac{5}{6}}}{P^2(\zeta)\zeta^{n-1}}d\zeta\frac{P^2(\xi)\xi^{n-1}}{\left(\xi^2-1\right)^{\frac{1}{6}+\frac{n}{2}}}\right|d\xi\ .\label{GEst-1}
		\end{align}
		Since $\eta_0\leq\xi\leq\zeta\leq\eta\leq1$, we have $\frac{\xi^{n-1}}{\zeta^{n-1}}\leq1$. Together with \eqref{P1}, we conclude
		\begin{align}
			\textstyle{\int_{\eta_0}^{\eta}\left|\int_{\xi}^{\eta}\frac{\left(\zeta^2-1\right)^{\frac{n}{2}-\frac{5}{6}}}{P^2(\zeta)\zeta^{n-1}}d\zeta\frac{P^2(\xi)\xi^{n-1}}{\left(\xi^2-1\right)^{\frac{1}{6}+\frac{n}{2}}}\right|d\xi
				= \int_{\eta_0}^{\eta}\left|\int_{\xi}^{\eta}\frac{\left(\zeta^2-1\right)^{\frac{n}{2}-\frac{5}{6}}}{\left(\xi^2-1\right)^{\frac{1}{6}+\frac{n}{2}}}\frac{P^2(\xi)}{P^2(\zeta)}\frac{\xi^{n-1}}{\zeta^{n-1}}d\zeta\right|d\xi}\notag\\
			\leq\textstyle{16\int_{\eta_0}^{\eta}\int_{\xi}^{\eta}\left|
				\frac{\left(\zeta^2-1\right)^{\frac{n}{2}-\frac{5}{6}}}{\left(\xi^2-1\right)^{\frac{1}{6}+\frac{n}{2}}}\right|
				d\zeta\ d\xi}
			\leq C_n |\eta-\eta_0|\label{GEst-2}
		\end{align}
		where $C_n$ depends on $n$ only.
		For $\delta_1 <\frac{1}{4C_n}\left|\frac{2}{(\sigma_1)^3}+\frac{2}{(c^*)^3}\right|^{-1}$
		one concludes
		\be\label{GContraction1}
		\left|\Gamma z_1(\eta)-\Gamma z_2(\eta)\right|\leq \textstyle{\frac{1}{4}}\|w_1-w_2\|_{C(\mathcal{I}_\delta; \mathds{R})} \leq \textstyle{\frac{1}{2}}\|z_1-z_2\|_{C(\mathcal{I}_\delta; \mathds{R})} \ \ \text{ for all $\eta\in\mathcal{I}_{\delta_1}$.}
		\ee
		\textbf{Preservation.} To see that $\Gamma:\ B_{\delta_2}\longrightarrow B_{\delta_2}$ for sufficiently small $\delta_2$,
		first note that $\Gamma z(\eta_{0})=z_0$. It suffices to show that there exists $\delta_2>0$, such that for all $\eta\in\mathcal{I}_{\delta_2}$
		\be\label{GMapInequality1}
		\sigma_1\leq c^*+w(\eta, \Gamma z(\eta))\leq\sigma_2.
		\ee
		Here,
		\begin{align}
			w(\eta, \Gamma z(\eta))&=\textstyle{P(\eta)z_0+\alpha_1P(\eta)\int_{\eta_{0}}^{\eta}\frac{\left(\zeta^2-1\right)^{\frac{n}{2}-\frac{5}{6}}}{P^2(\zeta)\zeta^{n-1}}d\zeta}\notag\\
			&\textstyle{-P(\eta)\int_{\eta_{0}}^{\eta}\left[\int_{\xi}^{\eta}\frac{\left(\zeta^2-1\right)^{\frac{n}{2}-\frac{5}{6}}}{P^2(\zeta)\zeta^{n-1}}d\zeta\right]
				\left[P(\xi)\left(\xi^2-1\right)^{-\frac{1}{6}-\frac{n}{2}}\xi^{n-1}\right]f(w(\xi, z))d\xi}\notag.
		\end{align}
		Let $\rho =|\alpha_1|$. Using \eqref{P1} we note that
		\begin{align}
			&\textstyle{\left|c^*+P(\eta)z_0+\alpha_1P(\eta)\int_{\eta_{0}}^{\eta}\frac{\left(\zeta^2-1\right)^{\frac{n}{2}-\frac{5}{6}}}{P^2(\zeta)\zeta^{n-1}}d\zeta\right|
				\leq\frac{2}{3}\sigma_2}\label{GEst0}
		\end{align}
		for $\sigma_2=\frac{3}{2}\left(c^*+2|z_0|+\frac{8(1-\eta_{0}) \rho (2|1-\eta_{0}|)^{\frac{n}{2}-\frac{5}{6}}}{\eta_{0}^{n-1}}\right)$.
		Because the integrand is bounded,
		\begin{align}
			&\textstyle{\left|P(\eta)\int_{\eta_{0}}^{\eta}\left[\int_{\xi}^{\eta}\frac{\left(\zeta^2-1\right)^{\frac{n}{2}-\frac{5}{6}}}{P^2(\zeta)\zeta^{n-1}}d\zeta\right]
				\left[P(\xi)\left(\xi^2-1\right)^{-\frac{1}{6}-\frac{n}{2}}\xi^{n-1}\right]f(w)d\xi\right| \leq C\left|\eta-\eta_{0}\right|}
			\label{GEst1}
		\end{align}
		for a  constant $C$ only depending on $\sigma_1$, $\sigma_2$ and $n$. 
		If $\delta_2$ is sufficiently small, we conclude from \eqref{GEst0} and \eqref{GEst1}  that for $\eta\in\mathcal{I}_{\delta_2}$
		\be\label{GEst3}
		w(\eta, \Gamma z(\eta))\leq \sigma_2.
		\ee
		Similarly, for {$\sigma_1>0$} with
		$P(\eta_{0})z_0\geq-c^*+2\sigma_1$,
		{continuity of $P$ implies
			\be\label{GInequality3-1}
			c^*+P(\eta)z_0\geq\textstyle{\frac{3}{2}}\sigma_1
			\ee
			for $\eta \in [\eta_0,\eta_0+\tilde{\delta}]$, $\tilde{\delta}$ depending on 
$\sigma_1$. For all $\eta$ in this intervall,} the boundedness of the integrand implies
		\begin{align}
			\left|\alpha_1P(\eta)\int_{\eta_{0}}^{\eta}\frac{\left(\zeta^2-1\right)^{\frac{n}{2}-\frac{5}{6}}}{P^2(\zeta)\zeta^{n-1}}d\zeta\right|\leq
			C|\eta-\eta_{0}|\label{GInequality3-2},
		\end{align}
			for a $C>0$ depending only on $\alpha_1$ and $n$.
			Hence, for sufficiently small $\delta_3$ 
			we deduce in combination with \eqref{GEst1} that for  $\eta\in\mathcal{I}_{\delta_3}$	
			\begin{equation}
				c^*+w(\eta, \Gamma z(\eta))\geq \sigma_1 .\label{GInequality3-3}
			\end{equation}
			
			We conclude that $\Gamma : B_\delta \to B_\delta$ is a contraction for $\delta = \min\left\{\delta_1, \delta_2, \delta_3\right\}$. According to the Banach fixed point theorem, $\Gamma$ has a unique fixed point, and thus the integral equation \eqref{IntegralSelfSimilarEquation} has a unique solution in $\mathcal{I}_{\delta}$.
			This concludes the proof of Lemma \ref{ExistenceNear1-G}.
		\end{proof}
		\subsection{Local Existence of Unique Self-Similar Solution for $\eta_0 = 0$}\label{LocExistSec3}
		
		We describe how to adapt the arguments in Section \ref{LocExistSec2} to the singular point $\eta_0 = 0$. In this case we use \eqref{IntegralSelfSimilarEquation} with $Q=\overline{P}$, which is smooth near $\eta=0$ and satisfies $\overline{P}(0)=1$. From $\frac{d\overline{P}}{d\eta}(0)=0$,
		the initial values \eqref{InitialValue0} for $v$ and the transformation \eqref{Transformation2} imply
		\be\label{InitialValue3-1}
		z(0)=c-c^*,\quad \frac{dz}{d\eta}(0)=0.
		\ee
		In particular, $\alpha_1=0$. The result is
		\begin{lemma}\label{ExistenceNear0-G}
			There exists a $\delta>0$,  depending only on the initial value \eqref{InitialValue3-1} and  on  $n$, such that the integral equation \eqref{IntegralSelfSimilarEquation} has a unique solution in $\mathcal{I_\delta} = [0, \delta]$.
		\end{lemma}
		\begin{proof}
			For given $\delta>0$, to be fixed later, we denote $\mathcal{I}_\delta = \left[0,\delta\right]$ and
			\be\label{GNBD}
			B_\delta=\left\{z\in C\left(\mathcal{I}_\delta; \mathds{R}\right):\ z(0)=z_0,\ \sigma_1\leq c^*+w(\eta, z)\leq\sigma_2\right\}.
			\ee
			Here, $\sigma_1$ and $\sigma_2$ are positive constants depending on the initial values \eqref{InitialValue3-1}. \\Consider the nonlinear operator $\Gamma$ on $B_\delta$ given by  \eqref{IntegralSelfSimilarEquation} with $\alpha_1=0$,
			\begin{align}
				\Gamma z(\eta)=\textstyle{z_0-\int_{0}^{\eta}\left[\int_{\xi}^{\eta}\frac{\left(\zeta^2-1\right)^{\frac{n}{2}-\frac{5}{6}}}{\overline{P}^2(\zeta)\zeta^{n-1}}d\zeta\right]\left[\overline{P}(\xi)\left(\xi^2-1\right)^{-\frac{1}{6}-\frac{n}{2}}\xi^{n-1}\right]f(w(\xi, z))d\xi}\label{GMap}.
			\end{align}
			The proof of Lemma \ref{ExistenceNear0-G} closely follows the proof of Lemma \ref{ExistenceNear1-G}. We describe the necessary adaptations.\\
			
			\noindent\textbf{Contraction.}
			It is easy to see that $\frac{1}{2}\leq \overline{P}\leq 2$ implies an estimate of the double integral in $\Gamma$ corresponding to \eqref{GEst-2}. Since \eqref{GEst-0} and \eqref{GEst-1} hold verbatim as in the proof of Lemma \ref{ExistenceNear1-G}, there again exists $\delta_1>0$, depending on the initial value \eqref{InitialValue3-1} and $n$, such that the contraction estimate \eqref{GContraction1} holds for all $\eta\in\mathcal{I}_{\delta_1}$.
			
			\noindent \textbf{Preservation.}
			The upper and lower bounds for $\overline{P}$ again imply an estimate of the double integral in $\Gamma$ corresponding to \eqref{GEst1}. Since $\alpha_1=0$, estimate \eqref{GEst0} is replaced by the observation
			$|c^* + \overline{P}(\eta) z_0| \leq c^* + 2 |z_0| =: \frac{2}{3} \sigma_2$.
			Together, these two estimates show as in Lemma \ref{ExistenceNear1-G} that there exists a $\delta_2>0$, depending on the initial value \eqref{InitialValue3-1} and the dimension $n$, such that \eqref{GEst1} holds. Then the upper bound \eqref{GEst3} again holds for all $\eta\in \mathcal{I}_{\delta_2}$. The lower bound \eqref{GInequality3-3} follows verbatim  for $\eta\in \mathcal{I}_{\delta_3}$ from \eqref{GEst1} and the observation
			$c^* + \overline{P}(\eta) z_0 \geq 2 \sigma_1$.\\
			The proof of Lemma \ref{ExistenceNear0-G} concludes by setting $\delta=\min\{\delta_1, \delta_2, \delta_3\}$.
		\end{proof}
		\subsection{Local Existence of a Family of Self-Similar Solutions for $\eta_0=1$}\label{LocExistSec4}
		When $\eta_0=1$, the initial condition $\frac{dv}{d\eta}(1)$ is determined by $v(1)$ from the differential equation \eqref{SelfSimilarEquation}, and the integral formulation \eqref{IntegralSelfSimilarEquation} contains a free parameter $\beta\in\mathds{R}$.
		As in Section \ref{LocExistSec2}, we consider \eqref{IntegralSelfSimilarEquation} with $Q=P$.
		Analogous to Lemma \ref{ExistenceNear1-G}, we obtain the local existence of a unique solution $v=v(\eta; [v(1), \beta])$
		to \eqref{IntegralSelfSimilarEquation} for given $\beta$:
		\begin{lemma}\label{ExistenceCross1-G}
			For any $\beta\in\mathds{R}$ there exists a $\delta>0$, depending on the initial value $v(1)$, on $\beta$ and on $n$, such that the integral equation \eqref{IntegralSelfSimilarEquation} has a unique solution in $\mathcal{I}_\delta=[1,\ 1+\delta]$.
		\end{lemma}
		\begin{proof}
			For given $\delta>0$, to be fixed later, we denote $\mathcal{I}_\delta = \left[1,\ \min\{1+\delta,\ 2\}\right]$ and
			\[B_\delta=\left\{z\in C\left(\mathcal{I}_\delta; \mathds{R}\right):\ z(1)=z_0,\ \sigma_1\leq c^*+w(\eta, z)\leq\sigma_2\right\} . \]
			Here, $\sigma_1$ and $\sigma_2$ are positive constants depending on the initial values $z_1$ and $\beta$. \\Consider the nonlinear operator $\Gamma$ on $B_\delta$ given by  \eqref{IntegralSelfSimilarEquation}, 
			\begin{equation}\label{IntBeyond1}
				\Gamma z(\eta)=z_0+ \textstyle{\beta\int_{1}^{\eta}\frac{\left(\zeta^2-1\right)^{\frac{n}{2}-\frac{5}{6}}}{{P}^2(\zeta)\zeta^{n-1}}d\zeta-\int_{1}^{\eta}\int_{\xi}^{\eta}\frac{\left(\zeta^2-1\right)^{\frac{n}{2}-\frac{5}{6}}}{{P}^2(\zeta)\zeta^{n-1}}d\zeta\frac{{P}(\xi)\xi^{n-1}}{\left(\xi^2-1\right)^{\frac{1}{6}+\frac{n}{2}}}f(w(\xi, z)) d\xi.}
			\end{equation}
{Note that the term involving $\alpha_1$ in \eqref{IntegralSelfSimilarEquation} vanishes when $\eta_0=1$.}
			The proof of Lemma \ref{ExistenceCross1-G} closely follows the proof of Lemma \ref{ExistenceNear1-G}, with a new term involving $\beta$. We here describe the necessary adaptations.\\
			
			\noindent\textbf{Contraction.}
			The double integral in the definition of $\Gamma$ is identical to the proof of Lemma \ref{ExistenceNear1-G}. The contraction argument is therefore unchanged, using $\frac{1}{2}\leq P \leq 2$.	
			
			\noindent \textbf{Preservation.}
			As in the contraction argument, all estimates for the double integral in $\Gamma$ follow by using $\frac{1}{2}\leq P \leq 2$. For the term involving $\beta$, the inequality \eqref{GEst0} is valid with $\alpha_1 := \beta$ and $\rho := |\beta|$ in the definition of $\sigma_2$. As in the proof of Lemma \ref{ExistenceNear1-G}, these two estimates show  the upper bound \eqref{GEst3}  for all $\eta\in \mathcal{I}_{\delta_2}$. The lower bound \eqref{GInequality3-3} follows verbatim  for $\eta\in \mathcal{I}_{\delta_3}$, from \eqref{GEst1} and the observations
			$c^* + P(\eta) z_0 \geq 2 \sigma_1 $, respectively
			\[\textstyle{\left|\beta P(\eta)\int_{1}^{\eta}\frac{\left(\zeta^2-1\right)^{\frac{n}{2}-\frac{5}{6}}}{P^2(\zeta)\zeta^{n-1}}d\zeta\right|\leq8|\beta|\int_{1}^{\eta}\left|\left(\zeta^2-1\right)^{\frac{n}{2}-\frac{5}{6}}\right|d\zeta
				\leq C|\eta-1|}.\]
			The proof of Lemma \ref{ExistenceCross1-G} concludes by setting $\delta=\min\{\delta_1,\delta_2,\delta_3\}$, which depends on $z_0$, $\beta$ and the dimension $n$.
		\end{proof}

\vskip2pc

\bibliographystyle{RS}
\bibliography{Bibliographyf.bib}

\end{document}